\documentclass[11pt]{article}
\usepackage{tgtermes}
\usepackage[english]{babel}
\usepackage[utf8]{inputenc}
\usepackage{fancyhdr}
\usepackage{amssymb}
\usepackage{amsmath}
\usepackage{dirtytalk}
\usepackage[margin=1in]{geometry}
\usepackage{bm}
\usepackage{comment}
\usepackage{amsfonts}
\usepackage{float}
\usepackage{appendix}
\usepackage[affil-it]{authblk}
\usepackage{textcomp}
\usepackage{subcaption}
\usepackage{amsthm}
\usepackage{enumitem}
\usepackage{indentfirst}
\usepackage{csquotes}
\usepackage{biblatex} 
\usepackage{hyperref}
\hypersetup{colorlinks,allcolors=black}
\usepackage{xcolor}
\usepackage[ruled]{algorithm2e}
\SetKwInput{kwAssert}{Assert}
\DeclareMathOperator{\SheafHom}{\mathcal{H\kern -3pt o\kern -2pt m\kern -1pt}}

\SetKwComment{Comment}{/* }{ */}
\DeclareNameAlias{default}{family-given}

\numberwithin{equation}{section}

\theoremstyle{definition}
\newtheorem{theorem}{Theorem}[section]

\newtheorem{lemma}[theorem]{Lemma}
\newtheorem{proposition}[theorem]{Proposition}
\newtheorem{definition}[theorem]{Definition}

\newtheorem{example}[theorem]{Example}
\newtheorem{corollary}[theorem]{Corollary}

\theoremstyle{remark}
\newtheorem*{remark}{Remark}

\usepackage{graphicx}
\graphicspath{ {./images/} }

\title{\textbf{Special Points Arising From Faithful Metacyclic and Dicyclic Galois Covers of the Projective Line}}

\author{Brian Yang}
\affil{California Institute of Technology}
\date{\today}

\begin{document}

\maketitle

\begin{abstract}
Within the Schottky problem, the study of special subvarieties of the Torelli locus has long been of great interest. We describe a representation-theoretic criterion for a Jacobian variety arising from a $G$-Galois cover of $\mathbb{P}^1$ branched at $3$ points to have complex multiplication (CM). For $G$ faithful metacyclic or dicyclic, we classify all such covers with Galois group $G$, identifying those that have CM. We compute the CM-field and type of Jacobian varieties arising from these covers, applying the representation theory of $G$ over $\mathbb{Q}$ and $\mathbb{Q}(\zeta_4)$. In particular, symplectic irreducible representations of $G$ are afforded by the Jacobian variety in the dicyclic case, giving rise to new examples of CM abelian varieties.
\end{abstract}

\section{Introduction}
\label{Introduction}

Let $\mathsf{M}_g$ be the moduli space of smooth complex algebraic curves of genus $g$, $\mathsf{A}_g$ be the moduli space of principally polarized abelian varieties of dimension $g$ over $\mathbb{C}$, and $j : \mathsf{M}_g \rightarrow \mathsf{A}_g$ be the Torelli morphism, taking $[C] \mapsto [\text{Jac}(C)]$ (isomorphism classes of curves to isogeny classes of Jacobians). Define $\mathsf{T}_g^{\circ} := j(\mathsf{M}_g)$ the open Torelli locus and its Zariski closure $\mathsf{T}_g := \overline{\mathsf{T}_g^{\circ}}$ the Torelli locus. Thus, the study of the open Torelli locus in Siegel varieties $\mathsf{A}_g$ is equivalent to the study of abelian varieties that arise as Jacobians of algebraic curves, a problem of great interest in algebraic and arithmetic geometry originally posed by Schottky.

We begin with the following exposition, following Freidani et. al~\cite{frediani2014shimura}. For $r \ge 3$, an $(r - 3)$-dimensional family of Galois covers of $\mathbb{P}^1$ is determined by a finite group $G$, a number of branching points $r \ge 3$, and a spherical system of generators (SSG) $\bm{x} = (x_1, \dots, x_r)$ of $G$ (cf. Section~\ref{Background}). Thus, the pair $(G, \bm{x})$ is called a \textit{monodromy datum} and describes the corresponding family of $G$-Galois covers. Each pair $(G, \bm{x})$ describes an $(r - 3)$-dimensional irreducible algebraic subvariety $\mathsf{M}(G, \bm{x}) \subseteq \mathsf{M}_g$, and hence via the Torelli morphism an $(r - 3)$-dimensional subvariety $\mathsf{Z}(G, \bm{x}) \subseteq \mathsf{A}_g$. It is natural to classify those data $(G, \bm{x})$ in which $\mathsf{Z}(G, \bm{x})$ is a \textit{special subvariety} of $\mathsf{A}_g$, as these special subvarieties contain a \textit{dense set} of CM points. Frediani et al. provide a sufficient criterion for classifying special subvarieties by considering the representation $\mathfrak{S}^2 H^0(C, \omega_C)$ of $G$~\cite[Theorem 3.9]{frediani2014shimura}: if
\begin{equation}
\label{eq:1.1}
    N := \dim (\mathfrak{S}^2 H^0(C, \omega_C))^G
\end{equation}
is $r - 3$, then $\mathsf{Z}(G, \bm{x})$ is special. We are concerned only with the application of~\eqref{eq:1.1} to the case of $r = 3$ branching points (i.e., trivial families of Galois covers). In the sequel, we always limit ourselves to this case. Then, $N = 0$ is a sufficient criterion for the $\mathsf{Z}(G, \bm{x})$ to be a \textit{special point}, hence a sufficient criterion for the Jacobian corresponding to $\mathsf{Z}(G, \bm{x})$ to admit CM.

To this end, the project aims to classify $G$-Galois covers $C \rightarrow \mathbb{P}^1$ with $3$ branching points, identifying those that are special with $N = 0$ . It is well-known that all cyclic and abelian covers of $\mathbb{P}^1$ with $3$ branching points give rise to CM Jacobian varieties, e.g. see~\cite[Lemma 3.1]{li2018newton}. Building off of the previous work in~\cite{li2018newton}, we consider the situation where the Galois group $G$ is one of the following families of non-abelian metacyclic groups:
\begin{enumerate}[nolistsep, label={(\roman{enumi})}]
    \item (Faithful metacyclic groups) $G \simeq (\mathbb{Z}/q\mathbb{Z}) \rtimes (\mathbb{Z}/n\mathbb{Z})$, $q$ prime, and the action $\mathbb{Z}/n\mathbb{Z} \rightarrow \text{Aut}(\mathbb{Z}/q\mathbb{Z})$ is faithful.
    \item (Dicyclic groups) $G \simeq (\mathbb{Z}/q\mathbb{Z}) \rtimes_2 (\mathbb{Z}/4\mathbb{Z})$, $q$ odd prime, and the action $\mathbb{Z}/4\mathbb{Z} \rightarrow \text{Aut}(\mathbb{Z}/q\mathbb{Z})$ has kernel isomorphic to $\mathbb{Z}/2\mathbb{Z}$.
\end{enumerate}
The CM of Jacobians arising from faithful metacyclic covers was previously investigated by Carocca et al.~\cite[Section 3]{Carocca2011}, whose results we revisit, refine, and extend using our perspective.

We wish to determine which monodromy data satisfy $N = 0$ (and thus are special), applying the representation theory of $G$. To give a more explicit presentation of the criterion of~\eqref{eq:1.1}, we make use of the \textit{Frobenius-Schur indicator} (cf.~\eqref{eq:frobschur}), which is used to classify the irreducible representations of $G$ over the real numbers. In particular, recall that a complex irreducible character $\chi \in \text{Irr}(G)$ is called \textit{orthogonal}, \textit{complex-valued}, or \textit{symplectic} if its Frobenius-Schur indicator $\iota_{\chi}$ is $1, 0$, or $-1$, respectively~\cite[Theorem 13.1]{Huppert1998}. Then, we obtain the following consequence of~\cite[Section 2]{frediani2014shimura}:
\begin{proposition}[See Proposition~\ref{prop:Nformula}, Corollary~\ref{corr:3.4}]
\label{prop:1.1}
Let $C \rightarrow \mathbb{P}^1$ be a Galois cover branched at $3$ points with datum $(G, \bm{x})$. Given any character $\chi \in \text{Irr}(G)$, let $\chi^{\ast} = \overline{\chi}$ be its dual character. Suppose $\mu_{\chi}$ is the multiplicity of $\chi$ in $H^0(C, \omega_C)$. Then,
\begin{equation}
    N = \frac12 \sum_{\chi \in \text{Irr(G)}} \mu_{\chi}(\mu_{\chi^{\ast}} + \iota_{\chi}).
\end{equation}
In particular, $N = 0$ (and so $\mathsf{Z}(G, \bm{x})$ is special) if and only if all of the following conditions are satisfied:
\begin{enumerate}[nolistsep, label={(\roman{enumi})}]
    \item For all orthogonal characters $\chi \in \text{Irr}(G)$, we have $\mu_{\chi} = 0$.
    \item For all complex-valued characters $\chi \in \text{Irr}(G)$, we either have $\mu_{\chi} = 0$ or $\mu_{\chi^{\ast}} = 0$.
    \item For all symplectic characters $\chi \in \text{Irr}(G)$, we have $\mu_{\chi} \in \{0, 1\}$.
\end{enumerate}
\end{proposition}
We also develop methods to compute the CM-field of Jacobian varieties associated to special monodromy data with $N = 0$. To do so, we utilize the \textit{isotypic decomposition} of complex abelian varieties with $G$-action to describe the endomorphisms of the Jacobian varieties arising from $G$-Galois covers~\cite[Theorem 2.2]{Lange2004}. We first settle the situation where $H^0(C, \omega_C)$ contains complex-valued irreducible characters $\chi$, i.e. $\iota_{\chi} = 0$. Here, the Proposition~\ref{prop:jacobiancm} generalizes the results~\cite[Lemma 3.1]{li2018newton} and~\cite[Theorem 1]{Carocca2011}. Next, we consider the situation where $H^0(C, \omega_C)$ contains \textit{symplectic irreducible characters} $\chi$, i.e. $\iota_{\chi} = -1$ of positive multiplicity. These examples were not discussed in previous works, and are trickier to handle than the complex-valued characters. The main idea is to compare the endomorphism algebra of the Jacobian to $K[G]$, where $K \supset \mathbb{Q}$ is a suitable imaginary quadratic extension. See Theorem~\ref{prop:jacobiancmsymplectic} for a classification of the CM-field in this case.

For the first family of groups $G \simeq (\mathbb{Z}/q\mathbb{Z}) \rtimes (\mathbb{Z}/n\mathbb{Z})$, the SSG $\bm{x}$ has orders, or local monodromy, of either $\bm{m} = (q, n, n)$ or $\bm{m} = (n_1, n_2, n_3)$ where each $n_i$ divides $n$. Here, we are mainly interested in the case  of $\bm{m} = (q, n, n)$ (the second case is not interesting), which was first studied by~\cite[Section 3]{Carocca2011}. In Section~\ref{sec:4}, we prove the main classification result:
\begin{theorem}[See Proposition~\ref{prop:4.2}, Theorem~\ref{thm:4.6}]
\label{thm:1.1}
Suppose $G \simeq (\mathbb{Z}/q\mathbb{Z}) \rtimes (\mathbb{Z}/n\mathbb{Z})$ where $q$ is a prime. A $G$-Galois cover $C \rightarrow \mathbb{P}^1$ with local monodromy $\bm{m} = (q, n, n)$ is special with $N = 0$ if and only if $n = 2, 3$. In case $n = 3$, the Jacobian $\text{Jac}(C)$ has complex multiplication by $\mathbb{Q}(\zeta_q)$.
\end{theorem}
Using primitive central idempotents, we also compute the CM-type of these Jacobians. Of note is the straightforward consequence that the CM Jacobian $\text{Jac}(C)$ as defined in Theorem~\ref{thm:1.1} above is isogeneous to the Jacobian $\text{Jac}(C')$ arising from a $(\mathbb{Z}/q\mathbb{Z})$-cyclic cover $C' \rightarrow \mathbb{P}^1$ with a datum of $\bm{y} = (1, k, k^2)$.

For the second family of groups $G \simeq (\mathbb{Z}/q\mathbb{Z}) \rtimes_2 (\mathbb{Z}/4\mathbb{Z})$, we are mainly interested in the local monodromy $\bm{m} = (q, 4, 4)$. This family of groups gives rise to symplectic characters of nonzero multiplicity in the space $H^0(C, \omega_C)$. In Section~\ref{sec:5}, we prove:
\begin{theorem}[See Proposition~\ref{prop:6.4}, Theorem~\ref{thm:cmtypedicyclic}]
\label{thm:1.3}
Suppose $G \simeq (\mathbb{Z}/q\mathbb{Z}) \rtimes_2 (\mathbb{Z}/4\mathbb{Z})$ where $q$ is an odd prime. Any $G$-Galois cover $C \rightarrow \mathbb{P}^1$ with local monodromy $\bm{m} = (q, 4, 4)$ is special with $N = 0$, and the Jacobian $\text{Jac}(C)$ has complex multiplication by $\mathbb{Q}(\zeta_{4q})$.
\end{theorem}
As an extension of the above, we also study the situation where the Galois group is $G \simeq Q_8$; the quaternion group is closely tied to the family of dicyclic groups of order $4q$. By computing the CM-type (using a similar method as for faithful metacyclic covers), we see that both cases $G \simeq  (\mathbb{Z}/q\mathbb{Z}) \rtimes_2 (\mathbb{Z}/4\mathbb{Z})$ and $G \simeq Q_8$ give rise to new examples of CM Jacobian varieties. It is of note that the endomorphism algebras of Jacobian varieties arising from such groups $G$ are studied over the base field $\mathbb{Q}(\zeta_4)$, rather than $\mathbb{Q}$.

\subsubsection*{Organization of the Paper}

Section~\ref{Background} contains basic facts and definitions about Galois covers of $\mathbb{P}^1$, rational group algebras, and abelian varieties with endomorphisms.

Section~\ref{sec:Jacobians} reviews special subvarieties of $\mathsf{A}_g$ and the CM Jacobians arising from these special subvarieties. We study several criteria for determining if the Jacobian of a Galois cover $C \rightarrow \mathbb{P}^1$ is of CM-type, analyzing both the complex-valued and symplectic cases.

Sections~\ref{sec:4} and~\ref{sec:5} utilizes methods introduced in previous sections to work out the respective examples of Galois groups: $G \simeq (\mathbb{Z}/q\mathbb{Z}) \rtimes (\mathbb{Z}/n\mathbb{Z})$ and $G \simeq (\mathbb{Z}/q\mathbb{Z}) \rtimes_2 (\mathbb{Z}/4\mathbb{Z})$. We obtain new examples of Jacobians of CM-type.

\section{Notation and Background}
\label{Background}

\subsection{Galois Covers of \texorpdfstring{$\mathbb{P}^1$}{Lg} Branched at Three Points}
\label{sec:2.1}

The information about a Galois cover of $\mathbb{P}^1$ is encoded by its \textit{monodromy datum}.
\begin{definition}
\label{monodromydatum}
A \textit{monodromy datum} is a pair $(G, \bm{x})$, where $G$ is a finite group and $\bm{x} = (x_1, x_2, x_3)$ is an triple of nonidentity elements of $G$ such that $\prod_{i = 1}^r x_i = 1$ and $x_1, x_2, x_3$ generate $G$. We call $\bm{x}$ a \textit{spherical system of $3$ generators} (SSG) of $G$. The triple $\bm{m} = (m_1, m_2, m_3)$, where $m_i$ denotes the order of $x_i$, is called the \textit{local monodromy} of $(G, \bm{x})$.
\end{definition}
The following constructions relate Galois covers of $\mathbb{P}^1$ with their monodromy data; see~\cite[Section 2.3]{frediani2014shimura} for more details. Let $t := (t_1, t_2, t_3)$ be an triple of distinct points in $\mathbb{P}^1$. Set $U_t := \mathbb{P}^1 \setminus \{t_1, t_2, t_3 \}$ and pick a base point $t_0 \in U_t$. For a smooth projective curve $C$ and Galois cover $f : C \rightarrow \mathbb{P}^1$ with branch locus $t$, define $V = f^{-1}(U_t)$. Then, there is a surjective homomorphism onto the group of deck transformations: $\pi_1(U_t, t_0) \twoheadrightarrow G := \text{Aut}(f |_V)$. By fixing an identification $\pi_1(U_t, t_0) = \langle \gamma_1, \gamma_2, \gamma_3, \ | \ \gamma_1\gamma_2\gamma_3 = 1 \rangle$, the surjective homomorphism $\pi_1(U_t, t_0) \twoheadrightarrow G$ gives rise to a spherical system of $3$ generators $x_1, x_2, x_3$ of $G$; the order $m_i$ of $x_i$ is the local monodromy around $t_i$. Thus, a Galois cover of $\mathbb{P}^1$ with branch locus $t$ determines a monodromy datum. Recall that the Riemann existence theorem allows the reverse of the above process: a monodromy datum $(G, \bm{x})$ defines a (trivial) family of $G$-Galois covers of $\mathbb{P}^1$. Namely, for every pair $(G, \bm{x})$ we obtain a point $\mathsf{M}(G, \bm{x}) \subseteq \mathsf{M}_g$. Let $\mathsf{Z}(G, \bm{x}) \subseteq \mathsf{A}_g$ be (the closure of) the image of $\mathsf{M}(G, \bm{x})$ under the Torelli morphism; this is a point of $\mathsf{A}_g$ which corresponds to the Jacobian of the underlying $G$-Galois cover.

Fix a monodromy datum $(G, \bm{x})$, and let $t$ be the locus. For any branch point $t_i$, the point $P \in f^{-1}(t_i)$ has ramification index $e_P = m_i$ the local monodromy. Thus, $\# f^{-1}(t_i) = \frac{\# G}{m_i}$. The Riemann-Hurwitz formula computes the genus $g := g(C)$ of the curve as
\begin{equation}
\label{eq:RiemannHurwitz}
    g = 1 - \# G + \frac{\sum_{P \in C} (e_P - 1)}{2} = 1 - \# G + \frac12 \cdot \sum_{i = 1}^r  \frac{\# G}{m_i}(m_i - 1),
\end{equation}
depending only on the local monodromy.

We recall that inequivalent monodromy data may give rise to the same subvariety in $\mathsf{M}_g$. Namely, 
let $\mathbf{B}_3 = \langle \sigma_1, \sigma_2 \ | \ \sigma_1\sigma_2\sigma_1 = \sigma_2\sigma_1\sigma_2 \rangle$ be the braid group on $3$ letters. Given any datum $(G, \bm{x})$, The braid $\mathbf{B}_3$ acts on $\bm{x}$ by \textit{Hurwitz moves}, i.e.
\begin{equation}
    \sigma_i \cdot x_i = x_ix_{i + 1}x_i^{-1}, \quad \sigma_i \cdot x_{i + 1} = x_i, \qquad \sigma_i \cdot x_j = x_j, \ \text{for} \ j \neq i, i + 1,
\end{equation}
for any generator $\sigma_i$ of $\mathbf{B}_3$, and extending $\sigma \cdot \bm{x} = (\sigma \cdot x_1, \sigma \cdot x_2, \sigma \cdot x_3)$ for any arbitrary $\sigma \in \mathbf{B}_3$. Furthermore, $\text{Aut}(G)$ acts pointwise on $\bm{x}$. Two monodromy data $(G, \bm{x}), (G, \bm{x}')$ give rise to the same point in $\mathsf{M}_g$ if the images of $\bm{x}, \bm{x}' \in G \times G \times G$ are in the same orbit under this canonical action of $\text{Aut}(G) \times \mathbf{B}_3$. In this case, they are called \textit{Hurwitz equivalent}~\cite[Section 2.3]{frediani2014shimura}.

\subsection{Preliminaries in Representation Theory}\label{sec:representationtheory}

In the following section, let $K$ be a subfield of the complex numbers $\mathbb{C}$. The main applications in this paper are the cases where $K = \mathbb{Q}$ or $K$ is a totally imaginary quadratic extension of $\mathbb{Q}$. The results below may generally be extended to the case where $K$ is any field of characteristic $0$.

Maschke's theorem asserts that the group algebra $K[G]$ of a finite group $G$ is semisimple. In particular, since $K[G]$ is finite dimensional over $K$, we may write its Wedderburn-Artin decomposition
\begin{equation}
\label{eq:WedderburnDec}
    K[G] \simeq R_1 \times \dots \times R_s, \quad R_i = M_{n_i}(\Delta_i). 
\end{equation}
Here $n_i \ge 1$ are integers and $\Delta_i$ are division algebras of finite dimension over $K$. While the numbers $n_i$ are easy to compute via dimension counting arguments, it is generally much harder to draw conclusions on the division rings $\Delta_i$ themselves.

We may study $K[G]$ via the complex theory. To this end, let $\text{Irr}(G)$ be the usual irreducible characters over $\mathbb{C}$ (in what follows, when we do not specify the field in which a representation/character is irreducible, we mean it is irreducible over $\mathbb{C}$). For $\chi \in \text{Irr}(G)$, let $K(\chi)$ be the character field of $\chi$ with respect to $K$, and let $e(\chi) = \frac{1}{\# G}\sum_{x \in G}\chi(1)\chi(x^{-1})x$ is the corresponding primitive central idempotent in $\mathbb{C}[G]$. Then, $e_{K}(\chi) = \sum_{\sigma \in \text{Gal}(K(\chi)/K)} e(\sigma \circ \chi)$, the sum of the $K$-Galois conjugate primitive central idempotents, is a primitive central idempotent in $K[G]$, and $K[G]e_{K}(\chi)$ is the corresponding simple Wedderburn component in $K[G]$. In fact, a direct computation verifies that the character field $K(\chi)$ may be embedded as the center of $K[G]e_{K}(\chi)$~\cite[Proposition 1.4]{Yamada1974}. Thus, $K[G] e_{K}(\chi) \simeq M_n(\Delta)$ for some integer $n \ge 1$ and some division ring $\Delta$ with center $K(\chi)$.

Now, recall that the degree of $\Delta$ over its center $K(\chi)$ over its perfect square. Namely, the number $m := m_{K}(\chi)$ such that $m^2 = [\Delta : K(\chi)]$ is called the \textit{Schur index} of $\chi$ over $K$~\cite[Chapter 12.2]{Serre1977}. Sometimes, the Schur index refers to the simple component $M_n(\Delta)$. Note for a subfield $F$ of $\Delta$ containing its center $K(\chi)$ is maximal (with respect to inclusion) if and only if $[F : K(\chi)] = m$. By no means are maximal subfields $F$ unique. The equivalent representation-theoretic definition of the Schur index is as follows~\cite[Chapter 10]{Isaacs2006}:
\begin{proposition}
\label{prop:2.3}
Take $\chi \in \text{Irr}(G)$. Let $\mathcal{W}$ be the irreducible $K[G]$-module corresponding to the simple algebra $K[G] e_{K}(\chi)$ as defined above. Then, the character of $\mathcal{W}$ as a $K$-representation is
\begin{equation}
\label{eq:2.3}
    \chi_{\mathcal{W}} := m_{K}(\chi) \cdot \sum_{\sigma \in \text{Gal}(K(\chi) / K)} (\sigma \circ \chi).
\end{equation}
\end{proposition}
In particular, this illustrates how the character $\chi$, along with all its $K$-Galois conjugates, is \textit{associated} to the Wedderburn component $K[G] e_K(\chi)$. For completeness, we provide the simple proof of Proposition~\ref{prop:2.3}, as both the semisimple algebra and the representation-theoretic perspectives are important to our work.
\begin{proof}[Proof of Proposition~\ref{prop:2.3}]
The $K(\chi)$-vector space $\mathcal{W} \otimes_{K} K(\chi)$ admits a module structure over $K(\chi)[G]$ and its subring
\begin{equation}
    K(\chi)[G]e_{K}(\chi) \simeq \prod_{\sigma \in \text{Gal}(K(\chi) / K)} K(\chi)[G] e(\sigma \circ \chi).
\end{equation}
Let $V_{\sigma}$ be the irreducible $K(\chi)[G]$-module associated to $K(\chi)[G] e(\sigma \circ \chi)$. Note that each $K(\chi)[G] e(\sigma \circ \chi)$ is isomorphic to $M_n(\Delta)$~\cite[Proposition 1.5]{Yamada1974}. 

Let $F$ be a maximal subfield of $\Delta$. For any extension $E \supseteq F$, the $E$-vector space $V_{\sigma} \otimes_{K(\chi)} E$ is an isotypic module over the simple component $M_{mn}(E) \simeq M_n(\Delta) \otimes_{K(\chi)} E \simeq E[G]e(\sigma \circ \chi)$ of $E[G]$. It follows $V_{\sigma} \otimes_{K(\chi)} E$ is an isotypic sum of $m$ absolutely irreducible representations; its character (as an $E$-representation) is $m(\sigma \circ \chi)$. By Galois theory and dimension counting, one deduces $\mathcal{W} \otimes_{K} K(\chi)$ is the direct sum of $V_{\sigma \circ \chi}$ over all $\sigma \in \text{Gal}(K(\chi) /K)$. Conclude that the complex representation $\mathcal{W} \otimes_{K} \mathbb{C}$ affords the character $m \cdot \sum_{\sigma \in \text{Gal}(K(\chi)/K)} (\sigma \circ \chi)$, as requested.
\end{proof}
In the notation of the above proof, note there is an irreducible $E$-representation of $G$ whose character is $\chi$; we say that the field $E$ \textit{realizes} the character $\chi$. Conversely, any extension $F' \supseteq K(\chi)$ which realizes $\chi$ is a splitting field of $\Delta$, so that $F'$ is an extension of a maximal subfield $F$ of $\Delta$ with $F \supseteq K(\chi)$. Thus, $m_K(\chi)$ is the minimum degree of an extension of $K(\chi)$ in which $\chi$ may be realized. In particular, 
\begin{quote}
    \textit{The character $\chi$ may be realized as an irreducible representation over its character field $K(\chi)$ if and only if $m_K(\chi) = 1$.}
\end{quote}
If $\chi, \psi \in \text{Irr}(G)$ are $K$-Galois conjugate characters, then the Schur indices of $\chi, \psi$ are obviously equal over any field extension of $K$. We refer the reader to~\cite[Chapter 38]{Huppert1998} or~\cite[Chapter 10]{Isaacs2006} for other well-known elementary properties of Schur indices. In the sequel, we shall cite these properties as needed (although we avoid using deep results if not needed).

The dimensions $n_i$ of~\eqref{eq:WedderburnDec} are easy to compute. Recall that if $\mathcal{W}_i$ is the simple $K[G]$-module corresponding to $R_i$, then $\Delta_i^{\circ} = \hom_{K[G]}(\mathcal{W}_i, \mathcal{W}_i)$ and $n_i = \dim_{\Delta_i^{\circ}} \mathcal{W}_i$. Then, if $m_i$ is the Schur index over $K$ of any $\chi \in \text{Irr}(G)$ associated to $R_i$, then a simple dimension counting argument shows $m_i n_i = \chi(1)$.

\subsection{Abelian Varieties with Group Action}\label{sec:2.3}

Let $\Lambda \subseteq \mathbb{C}^g$ be a lattice such that $A = \mathbb{C}^g/\Lambda$ is a complex abelian variety. Recall that the ring of endomorphisms $\text{End}(A)$ is the same as the ring of linear transformations of $\mathbb{C}^g$ preserving the lattice $\Lambda$. Likewise, the $\mathbb{Q}$-algebra $\text{End}^0(A) := \text{End}(A) \otimes_{\mathbb{Z}} \mathbb{Q}$ is the $\mathbb{Q}$-algebra of linear transformations $\mathbb{C}^g \rightarrow \mathbb{C}^g$ preserving the $\mathbb{Q}$-vector space $\Lambda_{\mathbb{Q}} := \Lambda \otimes_{\mathbb{Z}} {\mathbb{Q}}$, known as the \textit{analytic representation}. The endomorphism algebra $\text{End}^0(A)$ is finite dimensional over $\mathbb{Q}$, semisimple, and determined by the isogeny class of $A$.

Now we discuss finite groups acting on $A$, our exposition following~\cite{Lange2004}. Suppose $G$ is a finite group acting on $A$, so that this induces an action of $\mathbb{Q}[G]$ on $A$. More precisely, we have a $\mathbb{Q}$-algebra homomorphism $\rho : \mathbb{Q}[G] \rightarrow \text{End}^0(A) := \text{End}(A) \otimes_{\mathbb{Z}} \mathbb{Q}$. Hence, there is a complex representation of $G$ on $\mathbb{C}^g$ and a rational representation of $G$ on $\Lambda_{\mathbb{Q}}$, the two representations compatible with the embedding $\Lambda_{\mathbb{Q}} \hookrightarrow \mathbb{C}^g$. Then, the endomorphism algebra $\text{End}^0(A)$ may be compared to the semisimple algebra $\mathbb{Q}[G]$. As in~\eqref{eq:WedderburnDec}, we put
\begin{equation}
    \mathbb{Q}[G] \simeq R_1 \times \dots \times R_s, \quad R_i = M_{n_i}(\Delta_i)
\end{equation}
where $n_i \ge 1$ are integers and $\Delta_i$ are division algebras over $\mathbb{Q}$. For $1 \le i \le s$, let $e_i \in \mathbb{Q}[G]$ be the primitive central idempotent corresponding to $R_i$.
\begin{theorem}[{Isotypic Decomposition~\cite[Theorem 2.2]{Lange2004}}]\label{thm:isotypicdec} Let $A$ be a complex abelian variety with action by a finite group $G$ as above.
\begin{enumerate}[nolistsep, label={(\roman{enumi})}]
\item The abelian variety $A$ is isogenous to a product
\begin{equation}
    A \sim B_1^{n_1} \times \dots \times B_s^{n_s}
\end{equation}
where each $B_i$ is an abelian subvariety of $A$, given by $B_i = e_i A$.
\item For each $i$ such that $B_i$ is nonzero, $\Delta_i$ acts faithfully on $B_i$, and $M_{n_i}(\Delta_i)$ acts faithfully on $B_i^{n_i}$. Equivalently, $\text{End}^0(B_i), \text{End}^0(B_i^{n_i})$ contain isomorphic copies of $\Delta_i, M_{n_i}(\Delta_i)$, respectively.
\end{enumerate}
\end{theorem}
\begin{proof}
For (i), see~\cite[Proposition 1.1]{Lange2004} and~\cite[Proposition 2.1]{Lange2004}. (ii) is a straightforward consequence of (i) whose proof is omitted, although we will see similar ideas in the proof of Theorem~\ref{prop:jacobiancmsymplectic}.
\end{proof}
In the notation of the above theorem, each $A_i := B_i^{n_i}$ is called an \textit{isotypic component}.
\begin{remark}
Unlike the Poincar\'e irreducibility lemma, Theorem~\ref{thm:isotypicdec} isotypic decomposition is not necessarily an isogeny decomposition of $A$ into simple abelian varieties. Some $B_i$'s may not be simple, and in fact, some $B_i$'s may be zero.
\end{remark}

\subsection{Complex Multiplication of Abelian Varieties}
\label{sec:2.4}

We will now review the basic theory of complex multiplication (CM), following~\cite{jsmilne}. Suppose $A$ is a $g$-dimensional abelian variety. The ring $\text{End}^0(A)$ acts faithfully on the $2g$-dimensional $\mathbb{Q}$-vector space $H^1(A, \mathbb{Q})$, the first singular cohomology group. Thus, any maximal subfield of $\text{End}^0(A)$ is of degree at most $2g$ over $\mathbb{Q}$.
\begin{definition} Let $A$ be a $g$-dimensional abelian variety as above.
\begin{enumerate}[nolistsep, label={(\roman{enumi})}]
\item A field $K$ is called a \textit{CM-field} if it is an imaginary quadratic extension of a totally real number field.
\item The abelian variety $A$ is said to admit \textit{complex multiplication} (CM) if $\text{End}^0(A)$ contains a \textit{CM-field} $K$ of degree $2g$ over $\mathbb{Q}$. In this case, we say $A$ has CM by $K$.

\item A product of abelian varieties $\prod_i A_i$ is said to be CM by the \'etale $\mathbb{Q}$-algebra $\prod_i K_i$ if each $A_i$ has CM by the field $K_i$.
\end{enumerate}
\end{definition}

Suppose $A$ is an abelian variety with CM by $K$, so let $\mathfrak{i} : K \hookrightarrow \text{End}^0(A)$ be an embedding. Then, $H^1(A, \mathbb{Q})$ is a $1$-dimensional $K$-vector space. Consider a basis $\mathcal{B}$ of $H^1(A, \mathbb{Q}) \otimes_{\mathbb{Q}} \mathbb{C} = H^1(A, \mathbb{C})$ such that $a \in K$ acts via the diagonal matrix $\text{diag} \{\sigma_1(a), \dots, \sigma_{2g}(a)\}$, where $\sigma_i$ are the mutually distinct embeddings $K \hookrightarrow \mathbb{C}$. There are $g$ pairs of complex conjugate embeddings. The \textit{Hodge decomposition}
\begin{equation}
\label{eq:hodgedecompositiongeneral}
    H^1(A, \mathbb{C}) \simeq H^0(A, \Omega_A^1) \oplus \overline{H^0(A, \Omega_A^1)}
\end{equation}
preserves $\mathcal{B}$; in particular, $H^0(A, \Omega_A^1)$ induces a choice from each of the $g$ pairs of complex conjugate embeddings. This set of $g$ embeddings is denoted by $\Phi$, and it is called the \textit{CM-type} of the abelian variety $A$. We say that $(A, \mathfrak{i})$ gives rise to the \textit{CM-pair} $(K, \Phi)$. Conversely, any CM-pair $(K, \Phi)$ determines $(A, \mathfrak{i})$ up to isogeny~\cite[Proposition 3.12]{jsmilne}. If $A$ is a simple abelian variety, then $K = \text{End}^0(A)$ and so $A$ itself determines the CM-pair.

On the other hand, if a product of abelian varieties $\prod_i A_i$ has CM by the \'etale $\mathbb{Q}$-algebra $\prod_i K_i$, then the CM-type of $\prod_i K_i$ is a set $\Phi$ of nonzero $\mathbb{Q}$-algebra homomorphisms $\prod_i K_i \rightarrow \mathbb{C}$ given by the disjoint union of types $(K_1, \Phi_1) \sqcup (K_2, \Phi_2) \sqcup \dots$, where we may view $\Phi = \Phi_1 \sqcup \Phi_2 \sqcup \dots$. Each embedding $\sigma \in \Phi_j$, i.e. $\sigma : K_j \hookrightarrow \mathbb{C}$ naturally corresponds to a unique nonzero $\mathbb{Q}$-algebra homomorphism $\prod_i K_i \rightarrow \mathbb{C}$ that annihilates every $K_i$ for $i \neq j$. 

In Section~\ref{sec:Jacobians}, we shall see that the theory of CM closely relates to the theory of abelian varieties with group action.

\section{Jacobians of Galois Covers with Complex Multiplication}
\label{sec:Jacobians}

\subsection{Criterion for CM Jacobians}

The goal in this section is to review and reformulate more explicitly the criterion of Frediani et al. for a monodromy datum $(G, \bm{x})$ (as always, with $3$ branching points) to correspond to a special point in $\mathsf{A}_g$. Consider a $G$-Galois cover $C \rightarrow \mathbb{P}^1$ associated to the monodromy datum $(G, \bm{x})$, and let $g$ be the genus of $C$. The group $H^1(C, \mathbb{Q})$ is a $\mathbb{Q}[G]$-module of dimension $2g$. Moreover, $H^1(C, \mathbb{C}) = H^1(C, \mathbb{Q}) \otimes_{\mathbb{Q}} \mathbb{C}$ has a Hodge structure of type $(1, 0) + (0, 1)$,
\begin{equation}
\label{eq:hodgedecomposition}
    H^1(C, \mathbb{C}) = H^0(C, \omega_C) \oplus \overline{H^0(C, \omega_C)},
\end{equation}
where $H^0(C, \omega_C)$ is the $g$-dimensional vector space of holomorphic 1-forms (this is a specialized case of~\eqref{eq:hodgedecompositiongeneral}). Both $H^0(C, \omega_C), \overline{H^0(C, \omega_C)}$ are $\mathbb{C}[G]$-modules, uniquely determined up to isomorphism. Let $\varphi : G \rightarrow GL(H^0(C, \omega_C))$ be the representation (over $\mathbb{C}$) corresponding to this $\mathbb{C}[G]$-module structure, and let $\mathfrak{S}^2 \varphi : G \rightarrow GL(\mathfrak{S}^2 H^0(C, \omega_C))$ be the induced representation on the symmetric square. Let
\begin{equation}
    N := \dim (\mathfrak{S}^2 H^0(C, \omega_C))^G,
\end{equation}
be the dimension of the $G$-invariant subspace of $\mathfrak{S}^2 \rho$. Frediani et al. asserts that if $N = 0$, then the point $\mathsf{Z}(G, \bm{x}) \subset \mathsf{A}_g$ is special and the $g$-dimensional Jacobian $\text{Jac}(C)$ is CM~\cite[Theorem 3.9]{frediani2014shimura}. Note this is only a \textit{sufficient criterion} for $\mathsf{Z}(G, \bm{x})$ to be special, we have \textit{not} shown that monodromy data $(G, \bm{x})$ with $N > 0$ do not give rise to CM Jacobians (as remarked in~\cite[Section 3]{frediani2014shimura}, this is an interesting problem in its own right).

To compute $N$, we must classify the representation $H^0(C, \omega_C)$. We begin with the well-known formula of Chevalley-Weil. Given any $\chi \in \text{Irr}(G)$, let $\rho_{\chi}$ be the corresponding irreducible representation and $\mu_{\chi}$ the multiplicity of $\rho_{\chi}$ in $H^0(C, \omega_C)$. Let $E_{i, \alpha}$ be the number of eigenvalues of $\rho_{\chi}(x_i)$ equal to $\zeta_{m_i}^{\alpha}$, where $\zeta_n = e^{\frac{2\pi i}{n}}$ is a primitive $n$th root of unity.
\begin{theorem}[{Chevalley-Weil~\cite[Theorem 2.10]{frediani2014shimura}}]
\label{thm:cw}
Let $C \rightarrow \mathbb{P}^1$ be a Galois cover with monodromy datum $(G, \bm{x})$. Recall that $\bm{m} = (m_1, m_2, m_3)$ is the local monodromy of this datum. For irreducible characters $\chi \in \text{Irr}(G)$, let $\rho_{\chi}, E_{i, \alpha}$ be defined as above. The multiplicity $\mu_{\chi}$ of a degree $d_{\chi}$ irreducible character $\chi \in \text{Irr}(G)$ in $H^0(C, \omega_C)$ is given by the following formula:
\begin{equation}
    \mu_{\chi} = - d_{\chi} + \sum_{i = 1}^3 \sum_{\alpha = 0}^{m_i - 1} E_{i, \alpha} \left \langle - \frac{\alpha}{m_i} \right \rangle + \epsilon \quad \text{where} \quad \epsilon = \begin{cases} 1 & \text{$\chi$ is trivial} \\ 0 & \text{otherwise} \end{cases}.
\end{equation}
Here, $\langle q \rangle$ is the fractional part of any $q \in \mathbb{Q}$.
\end{theorem}
Now, let $\sigma : G \rightarrow GL(V)$ be any complex representation of $G$ with character denoted by $\chi_{\sigma}$, and let $\chi_{\mathfrak{S}^2 \sigma}$ be the character of the induced representation $\mathfrak{S}^2 \sigma$ on $\mathfrak{S}^2 V$. Then, for $x \in G$,
\begin{equation}
\label{eq:S2formula}
    \chi_{\mathfrak{S}^2 \sigma} = \frac12(\chi_{\sigma}(x)^2 + \chi_{\sigma}(x^2)).
\end{equation}
Using~\eqref{eq:S2formula}, the orthogonality relations, and the expansion $\chi_{\varphi} = \sum_{\chi \in \text{Irr}(G)} \mu_{\chi} \chi$, the result of~\cite[Section 2]{frediani2014shimura} is the formula 
\begin{equation}
\label{eq:S2rreformula}
    N = \langle \chi_{\mathfrak{S}^2 \varphi}, 1 \rangle = \frac{1}{2 \cdot \# G} \cdot \sum_{x \in G} \left( \sum_{\chi \in \text{Irr}(G)} \mu_{\chi} \chi(x) \right)^2  + \frac{1}{2 \cdot \# G} \cdot \sum_{x \in G}\sum_{\chi \in \text{Irr}(G)} \mu_{\chi} \chi(x^2).
\end{equation}
We would like to describe a simpler way to determine $N$. To this end, we make use the following result in representation theory over the real numbers due to Frobenius and Schur:
\begin{proposition}[Frobenius-Schur {\cite[Theorem 13.1]{Huppert1998}}]
For any irreducible character $\chi \in \text{Irr}(G)$, define
\begin{equation}
\label{eq:frobschur}
    \iota_{\chi} := \frac{1}{\# G} \sum_{x \in G} \chi(x^2).
\end{equation}
Then, $\iota_{\chi} \in \{-1, 0, 1\}$. Each case occurs in exactly the following situations:
\begin{enumerate}[nolistsep, label={(\roman{enumi})}]
    \item If $\chi$ is real-valued, and there exists a $\mathbb{R}$-representation of $G$ affording the character $\chi$, then $\iota_{\chi} = 1$.
    \item If $\chi$ attains non-real values, then $\iota_{\chi} = 0$.
    \item If $\chi$ is real-valued, and there does not exist a $\mathbb{R}$-representation of $G$ affording the character $\chi$, then $\iota_{\chi} = -1$.
\end{enumerate}
\end{proposition}
The formula~\eqref{eq:frobschur} is called the \textit{Frobenius-Schur indicator}. Given $\chi \in \text{Irr}(G)$, we call $\chi$ an \textit{orthogonal}, \textit{complex-valued}, or \textit{symplectic} character based on whether $\iota_{\chi}$ is $1, 0$, or $-1$, respectively. If $\chi$ is orthogonal or complex-valued, then it is clear $m_{\mathbb{R}}(\chi) = 1$. If $\chi = -1$, however, then $\mathbb{R}(\chi) = \mathbb{R}$, yet $\chi$ may only be realized over $\mathbb{C}$, so we have $m_{\mathbb{R}}(\chi) = 2$. This is an easy example of the determination of the Schur index for the field $K = \mathbb{R}$ (cf. Section~\ref{sec:representationtheory}).

Denote by $\chi^{\ast} = \overline{\chi}$ the dual character of $\chi$. The expression for $N$ in~\eqref{eq:S2rreformula} may be rewritten as
\begin{proposition}[Proposition~\ref{prop:1.1}]
\label{prop:Nformula}
Given a $G$-Galois cover $C \rightarrow \mathbb{P}^1$ with monodromy datum $(G, \bm{x})$, let $N = \dim (\mathfrak{S}^2 H^0(C, \omega_C))^G$. Then,
\begin{equation}
    N = \frac12 \sum_{\chi \in \text{Irr(G)}} \mu_{\chi}(\mu_{\chi^{\ast}} + \iota_{\chi})
\end{equation}
in terms of the multiplicities in $H^0(C, \omega_C)$ and Frobenius-Schur indicator of the complex irreducible characters of $G$.
\end{proposition}
\begin{proof}
The second term of~\eqref{eq:S2rreformula} simplifies readily as $\frac12 \sum_{\chi \in \text{Irr}(G)} \mu_{\chi} \iota_{\chi}$. The expansion $\chi_{\sigma}(x)^2$ in the first term of~\eqref{eq:S2rreformula} is a sum of terms of the form $\mu_{\chi} \mu_{\psi} \chi(x) \psi(x)$ for all ordered pairs $(\chi, \psi)$ of irreducible characters. Then, $\frac{\mu_{\chi} \mu_{\psi}}{\# G} \sum_{x \in G}\chi(x) \psi(x) = \mu_{\chi} \mu_{\psi} \langle \chi, \overline{\psi} \rangle$. The inner product $\langle \chi, \overline{\psi} \rangle$ equals $1$ if and only if $\chi = \overline{\psi}$, i.e. $\psi = \chi^{\ast}$; otherwise, it equals $0$. The requested formula follows.
\end{proof}
\begin{corollary}[Proposition~\ref{prop:1.1}]
\label{corr:3.4}
With the notation and assumptions of Proposition~\ref{prop:Nformula}, we have $N = 0$ if and only if all of the following conditions are satisfied:
\begin{enumerate}[nolistsep, label={(\roman{enumi})}]
    \item For all orthogonal characters $\chi \in \text{Irr}(G)$, we have multiplicity $\mu_{\chi} = 0$.
    \item For all complex-valued characters $\chi \in \text{Irr}(G)$, we either have $\mu_{\chi} = 0$ or $\mu_{\chi^{\ast}} = 0$.
    \item For all symplectic characters $\chi \in \text{Irr}(G)$, we have $\mu_{\chi} \in \{0, 1\}$.
\end{enumerate}
\end{corollary}
This is a clean representation theoretic statement for the criterion of $N = 0$. We use Corollary~\ref{corr:3.4} extensively in Sections~\ref{sec:4} and~\ref{sec:5} to identify the monodromy data giving rise to special subvarieties of $\mathsf{A}_g$.
\subsection{Determining the CM field}

The goal in the following series of results is to develop methods to compute the CM-algebra of Jacobians associated to special subvarieties with $N = 0$, relating the results to Corollary~\ref{corr:3.4}. Fix notation as in Section~\ref{sec:2.3}: denote by
\begin{equation}
\label{eq:3.8}
    \text{Jac}(C) \sim A_1 \times A_2 \times \dots \times A_s,
\end{equation}
the isotypic decomposition, where $A_i = e_i \text{Jac}(C)$  (cf. Theorem~\ref{thm:isotypicdec}). We also let $\mathcal{W}_1, \dots, \mathcal{W}_s$ be the rational irreducible representations of $G$, i.e. $\mathcal{W}_i$ corresponds to the simple component $R_i = M_{n_i}(\Delta_i)$ of $\mathbb{Q}[G]$. Since the Frobenius-Schur indicator is obviously preserved by the relation of Galois conjugacy, we may say that a component $R_i$ is \textit{orthogonal}, \textit{complex-valued} or \textit{symplectic} based on whether any $\chi \in \text{Irr}(G)$ associated to $R_i$ is orthogonal, complex-valued, or symplectic, respectively. Again take the usual setup: $C \rightarrow \mathbb{P}^1$ is a $G$-Galois cover with monodromy datum $(G, \bm{x})$, and for any $\chi \in \text{Irr}(G)$, $\mu_{\chi}$ is the multiplicity of the character $\chi$ in the complex representation $H^0(C, \omega_C)$.

First, we study a \say{generic} situation, i.e. we wish to directly extract CM fields from the Wedderburn components of $\mathbb{Q}[G]$. We will see that this is \say{possible} only for the complex-valued components of $\mathbb{Q}[G]$. Let us state the main result; note that we are making no initial assumptions on any Schur indices.
\begin{proposition}
\label{prop:jacobiancm}
Let $1 \le i \le s$, $\chi \in \text{Irr}(G)$ a character associated to $R_i$, and let $d_0$ the degree of any maximal subfield of $R_i = M_{n_i}(\Delta_i)$ over $\mathbb{Q}$. Then, 
\begin{enumerate}[nolistsep, label={(\roman{enumi})}]
    \item We have $d_0 \le 2 \dim A_i$, with equality if and only if $m_{\mathbb{Q}}(\chi) = 1$ and $\mu_{\chi} + \mu_{\chi^{\ast}} = 1$.
    \item If equality holds in (i), then $\chi$ is complex-valued, $\Delta_i = \mathbb{Q}(\chi)$ is a CM-field, and the isotypic component $A_i$ has complex multiplication by a composite field $F_0 \cdot \mathbb{Q}(\chi)$, where $F_0 \supset \mathbb{Q}$ is a totally real extension of degree $n_i$.
\end{enumerate}
\end{proposition}
\begin{proof}
Since the center of $M_{n_i}(\Delta_i)$ is $\mathbb{Q}(\chi)$, the degree of any maximal subfield of $M_{n_i}(\Delta_i)$ containing $\mathbb{Q}$ is exactly
\begin{equation}
    d_0 = [M_{n_i}(\Delta_i) : \Delta_i]^{\frac12}[\Delta_i : \mathbb{Q}(\chi)]^{\frac12}[\mathbb{Q}(\chi) : \mathbb{Q}] = n_i m_{\mathbb{Q}}(\chi)[\mathbb{Q}(\chi) : \mathbb{Q}];
\end{equation}
see~\cite[Proposition 1.3]{jsmilne}.

On the other hand, $\dim_{\mathbb{Q}} \mathcal{W}_i = \dim \chi \cdot m_{\mathbb{Q}}(\chi) [\mathbb{Q}(\chi) : \mathbb{Q}] = n_i m_{\mathbb{Q}}(\chi)^2 [\mathbb{Q}(\chi) : \mathbb{Q}]$.
Since $e_i H^1(C, \mathbb{Q})$ is isotypic: it is a direct sum of finitely many $\mathcal{W}_i$'s, $\dim_{\mathbb{Q}} e_iH^1(C, \mathbb{Q})$ is a positive integer multiple of $n_i m_{\mathbb{Q}}(\chi)^2 [\mathbb{Q}(\chi) : \mathbb{Q}]$. Thus,
\begin{equation}
\label{eq:3.10}
    d_0 \le \dim_{\mathbb{Q}} e_iH^1(C, \mathbb{Q}) = 2 \dim A_i
\end{equation}
with equality if and only if $m_{\mathbb{Q}}(\chi) = 1$ (equiv. $\Delta_i = \mathbb{Q}(\chi)$) and the multiplicity of $\mathcal{W}_i$ in $e_i H^1(C, \mathbb{Q})$ is $1$. However, by the Hodge decomposition~\eqref{eq:hodgedecomposition}, the multiplicity of $\mathcal{W}_i$ in $e_i H^1(C, \mathbb{Q})$ is $(\mu_{\chi} + \mu_{\chi^{\ast}}) / m_{\mathbb{Q}}(\chi)$, where the factor of $m_{\mathbb{Q}}(\chi)$ arises due to Proposition~\ref{prop:2.3}. The proof of (i) is complete.

Now assume the equality case of (i) holds. Recall that the action of $G$ on $C$ induces an action of $\mathbb{Q}[G]$ on $\text{Jac}(C)$; Theorem~\ref{thm:isotypicdec} implies $M_{n_i}(\Delta_i)$ is contained in $\text{End}^0(A_i)$. The assumption $\mu_{\chi} + \mu_{\chi^{\ast}} = 1$ forces $\chi$ to be a complex valued character (in case of real-valued characters, $\mu_{\chi} + \mu_{\chi^{\ast}}$ is even), so $\Delta_i = \mathbb{Q}(\chi)$ is a complex abelian extension. If $F_0$ is a totally real extension of $\mathbb{Q}$ of degree $m_i$ linearly disjoint to $\mathbb{Q}(\chi)$, then $A_i$ has CM by the composite $F := F_0 \cdot \mathbb{Q}(\chi)$, as $[F : \mathbb{Q}] = d_0$ and the choice of $F_0$-basis defines an embedding of $F$ into $\text{End}^0(A_i)$. This proves (ii).
\end{proof}
We observe that the following proposition in fact recovers~\cite[Lemma 3.1]{li2018newton} and~\cite[Corollary 3.7]{milescua}: the classification of CM algebras for Jacobians arising from cyclic and abelian Galois covers branched at $3$ points. Indeed, in the abelian situation, the equality case of Proposition~\ref{prop:jacobiancm} (i) always holds for all nonzero isotypic components.

Now, we restrict to the more-subtle case of symplectic characters in $H^1(C, \mathbb{Q}) \otimes_{\mathbb{Q}} \mathbb{C}$. The key idea will be to extend natural the $\mathbb{Q}[G]$-action (cf. Section~\ref{sec:2.3}) on the Jacobian $\text{Jac}(C)$, or some abelian subvariety $A_i$, to a $K[G]$-action, where $K \supset \mathbb{Q}$ is an imaginary quadratic extension. Recall that the Brauer-Speiser theorem~\cite[Corollary 1.8]{Yamada1974} states that any real-valued character $\chi \in \text{Irr}(G)$ satisfies $m_{\mathbb{Q}}(\chi) \in \{1, 2\}$. Thus, if $\chi$ is symplectic, then $\chi$ may not be realized over $\mathbb{Q}(\chi)$, so $m_{\mathbb{Q}}(\chi) = 2$. Furthermore, any extension $K \supset \mathbb{Q}$ in which $m_K(\chi) = 1$ must be nonreal (by $m_{\mathbb{R}}(\chi) = 2$) and of degree at least $2$ (see~\cite[Corollary 10.2]{Isaacs2006}). 

The main theorem we prove is
\begin{theorem}
\label{prop:jacobiancmsymplectic}
Let $R_i = M_{n_i}(\Delta_i)$ be a symplectic Wedderburn component of $\mathbb{Q}[G]$, and let $\chi \in \text{Irr}(G)$ be a character associated to $R_i$. Assume that $K \supset \mathbb{Q}$ is an imaginary quadratic extension such that $m_K(\chi) = 1$. Then,
\begin{enumerate}[nolistsep, label={(\roman{enumi})}]
    \item The natural $\mathbb{Q}[G]$-action on $A_i$ (cf. Section~\ref{sec:2.3}) extends to a $K[G]$-action on $A_i$, and in particular $\text{End}^0(A_i)$ contains $M_{n_i}(\Delta_i) \otimes_{\mathbb{Q}(\chi)} K = M_{2n_i}(K(\chi))$.
\end{enumerate}
Let $d_1$ be the degree of any maximal subfield of $M_{2n_i}(K(\chi))$ over $\mathbb{Q}$. Then,
\begin{enumerate}[nolistsep, label={(\roman{enumi})}]
    \setcounter{enumi}{1}
    \item We have $d_1 \le 2 \dim A_i$, with equality if and only if $\mu_{\chi} = 1$.
    \item If equality holds in (ii), then the isotypic component $A_i$ has complex multiplication by a composite field $F_0 \cdot K(\chi)$, where $F_0 \supset \mathbb{Q}$ is a totally real extension of degree $2n_i$.
\end{enumerate}
\end{theorem}
\begin{proof} 
The main non-trivial result here is (i). Let $\chi \in \text{Irr}(G)$ be any irreducible character associated to $M_{n_i}(\Delta_i)$. Then, the character of $\mathcal{W}_i$, as a $\mathbb{Q}$-representation, is $m_{\mathbb{Q}}(\chi) \cdot \Sigma(\chi) = 2 \cdot \Sigma(\chi)$, where $\Sigma(\chi) := \sum_{\sigma \in \text{Gal}(\mathbb{Q}(\chi)/\mathbb{Q})} (\sigma \circ \chi)$ is the sum of the $\mathbb{Q}$-Galois conjugates of $\chi$ (cf. Proposition~\ref{prop:2.3}). Note that $K(\chi)/\mathbb{Q}(\chi)$ is an imaginary quadratic extension, so that $[K(\chi) : K] = [\mathbb{Q}(\chi) : \mathbb{Q})]$: the $\mathbb{Q}$-Galois conjugacy class of $\chi$ remains the same upon moving to the larger field $K$.

Recall by definition $\text{End}^0(A_i)$ is the $\mathbb{Q}$-algebra of linear maps of $e_i H^0(C, \omega_C)^{\ast}$ into itself preserving $\Lambda'_{\mathbb{Q}} := e_i H^1(C, \mathbb{Q})^{\ast} \simeq \mathcal{W}_i$, where $\Lambda'$ is the sublattice of $\text{Jac}(C)$ induced by multiplication by $e_i$. Denote by $V_i$ the irreducible $K[G]$-representation whose character is $\Sigma(\chi)$. This is justified as $m_K(\chi) = 1$ and the above facts on the $K$-Galois conjugacy class of $\chi$. By the forgetful functor, we may view $V_i$ as an irreducible $\mathbb{Q}$-representation of degree $[K : \mathbb{Q}] \dim_{\mathbb{Q}(\zeta_4)} V_i$. The character of this representation may be computed:~\cite[Lemma 9.18]{Isaacs2006} implies that the character of $V_i$ as a $\mathbb{Q}$-representation is precisely $2 \cdot \Sigma(\chi)$. Hence, fix an identification $V_i = \mathcal{W}_i$ and thus an identification $\Lambda'_{\mathbb{Q}} = V_i$. This induces a $K[G]$-module structure on $\Lambda'_{\mathbb{Q}}$.

Now consider the natural embedding
\begin{equation}
    \mathcal{W}_i \hookrightarrow V_i \otimes_{\mathbb{Q}} \mathbb{R} \simeq V_i \otimes_{K} \mathbb{C}
\end{equation}
where the representations on the RHS are $\mathbb{C}$-isomorphic. Note that the action of $K$ on $V_i \otimes_{\mathbb{Q}} \mathbb{R}$ preserves the image of $\mathcal{W}_i$ and commutes with that of $G$. By tensoring with $\mathbb{R}$, the identification $\Lambda'_{\mathbb{Q}} = V_i$ induces an identification of real spaces $V_i \otimes_{\mathbb{Q}} \mathbb{R} = \Lambda' \otimes_{\mathbb{Z}} \mathbb{R} \simeq H^0(C, \omega_C)^{\ast}$ where the second isomorphism is by definition of the Jacobian. By Theorem~\ref{thm:cw}, the Chevalley-Weil formula, the complex structures of the above are also preserved. It follows that $A_i$ affords a $K[G]$-action extending the natural $\mathbb{Q}[G]$-action $\rho$ (cf.~\ref{sec:2.3}): denote this action by the map $\rho_K : K[G] \rightarrow \text{End}^0(A_i)$. Now, observe that $e_i$, which we originally defined for $\mathbb{Q}[G]$, is in fact the primitive central idempotent associated to the simple Wedderburn component $M_{n_i}(\Delta_i) \otimes_{\mathbb{Q}(\chi)} K = M_{2n_i}(K(\chi))$ of $K[G]$! Since $A_i = e_i \text{Jac}(C)$, the algebra $M_{2n_i}(K(\chi))$ acts faithfully on $A_i$, proving (i)\footnote{This is the same idea as the proof of Theorem~\ref{thm:isotypicdec} (ii).}.

Now, the proofs of (ii) and (iii) follow similar steps as that of Proposition~\ref{prop:jacobiancm}. Indeed, the field $K(\chi)$ is already CM by definition. Here, the Hodge decomposition~\eqref{eq:hodgedecomposition} implies $\mu_{\chi}$ is precisely the multiplicity of $\mathcal{W}_i$ in $e_i H^1(C, \mathbb{Q})$.
\end{proof}
A few remarks are now in order. Suppose that for a given index $1 \le i \le s$: either $A_i = 0$, or the equality case of Proposition~\ref{prop:jacobiancm} or Theorem~\ref{prop:jacobiancmsymplectic} applies. Thus, for $\chi \in \text{Irr}(G)$ associated to $R_i$
\begin{itemize}[nolistsep]
    \item If $\chi$ is real, then neither of the two equality cases above apply to the index $i$, and so $\mu_{\chi} = 0$.
    \item If $\chi$ is complex-valued and $A_i \neq 0$, then we must be in the equality case of Proposition~\ref{prop:jacobiancm}. Thus, $\mu_{\chi} + \mu_{\chi^{\ast}} \in \{0, 1\}$.
    \item If $\chi$ is symplectic, then we must be in the equality case of Theorem~\ref{prop:jacobiancmsymplectic}. Thus, $\mu_{\chi} \in \{0, 1\}$.
\end{itemize}
This implies that Proposition~\ref{prop:jacobiancm} and Theorem~\ref{prop:jacobiancmsymplectic} together describe a subset of all possible cases in which $N = 0$ cf. Corollary~\ref{corr:3.4}. However, we do not know whether there are cases in which $N = 0$ that are not covered by Proposition~\ref{prop:jacobiancm} or Theorem~\ref{prop:jacobiancmsymplectic}. For instance, it would be interesting to Galois covers arising from groups with Schur indices $m_{\mathbb{Q}}(\chi) > 2$. The smallest finite group $G$ with $\chi \in \text{Irr}(G)$ such that $m_{\mathbb{Q}}(\chi) > 2$ is the non-abelian group $(\mathbb{Z}/7\mathbb{Z}) \rtimes_3 (\mathbb{Z}/9\mathbb{Z})$ of order $63$ (where the action $\mathbb{Z}/9\mathbb{Z} \rightarrow \text{Aut}(\mathbb{Z}/7\mathbb{Z})$ has kernel isomorphic to $\mathbb{Z}/3\mathbb{Z}$); this group has an irreducible character of Schur index $3$. A generalization of this example for arbitrarily large Schur indices is given in~\cite[Example 38.19]{Huppert1998}. We did not find any Galois covers with $N = 0$ arising from such groups with larger Schur indices.

The problem of supplying an imaginary quadratic extension $K \supset \mathbb{Q}$ such that $m_K(\chi) = 1$ as in the situation of Theorem~\ref{prop:jacobiancmsymplectic} is also interesting. Certainly such an extension exists if $\mathbb{Q}(\chi) = \mathbb{Q}$. In the examples we work with in Section~\ref{sec:5}, we observe that the $4$th cyclotomic extension field $K = \mathbb{Q}(\zeta_4)$ always suffices.

\section{Classification of Faithful Metacyclic Covers}
\label{sec:4}

Let $K$ and $H$ be finite cyclic groups. We call a finite group $G$ a \textit{faithful metacyclic group} if $G \simeq K \rtimes H$ is a semidirect product where the underlying action of $H$ on $K$ is faithful, i.e. the map $H \rightarrow \text{Aut}(K)$ is injective. Here we study examples of $G$-Galois covers of $\mathbb{P}^1$ branched at $3$ points where $G$ is a faithful metacyclic group.

Let $G = D_n = \langle a, b \ | \ a^n = b^2 = 1, bab = a^{-1} \rangle$ be the dihedral group of order $2n$, By the Riemann-Hurwitz formula (cf.~\eqref{eq:RiemannHurwitz}), one immediately shows that any $G$-Galois cover with $3$ branching points is trivial. Now, given positive integers $q, n > 1$ such that $q$ is a prime and $n \mid q - 1$, we shall study in Section~\ref{sec:faithfulmetacyclic} the case $G = G_{q, n}$ is a faithful metacyclic groups of order $qn$ defined by
\begin{equation}
\label{eq:4.2}
    G_{q, n} \simeq (\mathbb{Z}/q\mathbb{Z}) \rtimes (\mathbb{Z}/n\mathbb{Z}) \simeq \langle a, b \ | \ a^q = b^n = 1, b^{-1} a b = a^k \rangle
\end{equation}
where $1 < k < q$ is of multiplicative order $n$ modulo $q$ (of course, $n = 2$ recovers the dihedral group $D_q$). We determine the cases giving rise to $N = 0$ special families and the CM-field and type of those cases.
\subsection{Groups of Order \texorpdfstring{$qn$}{Lg} with Faithful Semidirect Product}
\label{sec:faithfulmetacyclic}

Let $G := G_{q, n}$ be the metacyclic group defined as above. We begin by considering the possible local monodromy of these $G$-Galois covers:
\begin{proposition}
\label{prop:monodromyfaithfulmetacyclic}
Let $C \rightarrow \mathbb{P}^1$ be a $G$-Galois cover, branched at $3$ points, with monodromy datum $(G, \bm{x})$. Then, the local monodromy $\bm{m}$ is either $(q, n, n)$, or it is a $3$-tuple $\bm{m} = (n_1, n_2, n_3)$ where each $n_i$ divides $n$ which may be viewed as the local monodromy of a cyclic cover.
\end{proposition}
\begin{proof}
If one element $x_i$ in $\bm{x}$ is order $q$, that is, $x_i \in \langle a \rangle \setminus \{1\}$, then it is necessary that any other element $x_j$ is order $n$ (equiv. $x_j \in b^{\nu}\langle a \rangle$ with $\gcd(\nu, n) = 1$), in order for $x_i, x_j$ to generate $G$. Thus, assume that the local monodromy $\bm{m}$ is some permutation of $(q, n, n)$. Note: by an appropriate $\mathbf{B}_3 \times \text{Aut}(G)$-action we may assume $x_1 = a$ (and so $m_1 = q$). 

Otherwise, each $x_i$ is not contained in $\langle a \rangle$ (of course, this means $n \ge 3$). Considering $\bm{x}$ as a tuple of elements in $G/\langle a \rangle \simeq \mathbb{Z}/n\mathbb{Z}$ gives rise to a cyclic monodromy datum $(\mathbb{Z}/n\mathbb{Z}, \bm{y})$. The local monodromy of this cyclic cover is clearly a $3$-tuple of positive integers dividing $n$.
\end{proof}
\begin{remark}
In the $n = 2, 3, 4$ cases, there is a unique Hurwitz equivalence class where $\bm{m} = (q, n, n)$ or any permutation of these numbers.
\end{remark}

Carocca et al. proves via genus considerations that there exists a $G$-monodromy datum $\bm{x}$ (with $r \ge 3$ branching points) giving rise to a Jacobian $\text{Jac}(C)$ with complex multiplication by $\mathbb{Q}(\zeta_q)$ if and only if there are $r = 3$ branching points and $n = 3$~\cite[Proposition 3.1, Proposition 3.4]{Carocca2011}. The data here coincides with the local monodromy $\bm{m} = (q, n, n)$, i.e. the first case of Proposition~\ref{prop:monodromyfaithfulmetacyclic}. Using the results of Section~\ref{sec:Jacobians}, we present an alternative approach for the case of $3$ branching points, including a complete classification of the data $(G, \bm{x})$ giving rise to a special point with $N = 0$. This generalizes and extends the methods of~\cite{Carocca2011}.

In the case where the $G$-Galois cover $C \rightarrow \mathbb{P}^1$ has local monodromy $\bm{m}$ corresponding to a cyclic cover, the classification of special points and CM Jacobians is resolved by uninteresting and elementary computations and casework. As such, we do not further discuss this case in the paper.

\subsubsection*{Representation Theory of $G_{q, n}$ over $\mathbb{C}$ and $\mathbb{Q}$}

We briefly review the representation theory of $G$ over $\mathbb{C}$ and $\mathbb{Q}$, also introducing useful notation to be used later. First recall there are $n  + \frac{q - 1}{n}$ conjugacy classes of $G$, as follows. Indeed, for each $1 \le \nu < n$, the coset $b^\nu \langle a \rangle$ is a conjugacy class of $G$. There are $s := \frac{q - 1}{n}$ size $n$ conjugacy classes of $G$ contained in $\langle a \rangle$, indexed by the orbits of the multiplication by $k$ map in $(\mathbb{Z}/q\mathbb{Z})^{\ast}$.

Recall $G' = \langle a \rangle$, so $G/G' \simeq \mathbb{Z}/n\mathbb{Z}$. Thus, $G$ has $n$ complex linear characters $\psi_0, \dots, \psi_{n - 1}$ lifted from those of $\mathbb{Z}/n\mathbb{Z}$; assume $\psi_0$ is the trivial character. The family of $s$ degree-$n$ complex irreducible representations of $G$ are obtained as follows. Suppose $t_1, \dots, t_s$ is a set of representatives for the $k$-orbits in $(\mathbb{Z}/q\mathbb{Z})^{\ast}$, so the orbit corresponding to $t_i$ is precisely $\{t_i, kt_i, \dots, k^{n - 1}t_i\}$. Let $\theta_1, \dots, \theta_s$ be the linear representations of $\mathbb{Z}/p\mathbb{Z} \simeq \langle a \rangle$ corresponding to the map $1 \mapsto \zeta_q^{t_i}$. Then, the induced representations $\rho_i = \text{Ind}_{\langle a \rangle}^G(\theta_i), \ 1 \le i \le s$ are degree-$n$, irreducible (e.g. by~\cite[Proposition 8.25]{Serre1977}), and inequivalent. It is straightforward to construct a matrix representation of the $\rho_i$:
\begin{equation}
\label{eq:4.4}
    \rho_i : a \mapsto 
    \begin{bmatrix}
    \zeta_q^{t_i} & 0 & \cdots & 0 \\
    0 & \zeta_q^{kt_i} & \cdots & 0 \\
    \vdots & \vdots & \ddots & \vdots \\
    0 & 0 & \cdots & \zeta_q^{k^{n - 1}t_i}
    \end{bmatrix}, \quad b \mapsto
    \begin{bmatrix}
    0 & 0 & 0 & \cdots & 0 & 1 \\
    1 & 0 & 0 & \cdots & 0 & 0 \\
    0 & 1 & 0 & \cdots & 0 & 0 \\
    \vdots & \vdots & \vdots & \ddots & \vdots & \vdots \\
    0 & 0 & 0 & \cdots & 0 & 0 \\
    0 & 0 & 0 & \cdots & 1 & 0
    \end{bmatrix}.
\end{equation}
This is useful for determining the eigenvalues of the $\rho_i$.

Denote by $\chi_1, \dots, \chi_s$ the respective characters of $\rho_1, \dots, \rho_s$, and put $\chi := \sum_{i = 1}^s \chi_i$. It is well-known that the Schur index of every irreducible character of $G$ over any subfield $F \subseteq \mathbb{C}$ is $1$. For the induced degree-$n$ characters, one may appeal to~\cite[Lemma 10.8]{Isaacs2006} for an elementary proof. Let $\mathbb{Q}(\zeta_q^{(n)})$ be the unique subfield of $\mathbb{Q}(\zeta_q)$ of degree $n$, i.e. $\mathbb{Q}(\zeta_q^{(n)})$ generated by $\mathbb{Q}$ and $\sum_{\nu = 0}^{n - 1} \zeta_q^{k^\nu t_1}$. Note $\mathbb{Q}(\zeta_q^{(n)}) = \mathbb{Q}(\chi_i)$ for each $i$. Then, by the exposition in Section~\ref{sec:representationtheory}, the group algebra decomposition is
\begin{equation}
    \mathbb{Q}[G] \simeq \left( \prod_{d \mid n} \mathbb{Q}(\zeta_n) \right) \times M_n(\mathbb{Q}(\zeta_q^{(n)})).
\end{equation}
In particular, there is a irreducible $\mathbb{Q}$-representation of $G$ corresponding to $M_n(\mathbb{Q}(\zeta_q^{(n)}))$ whose character is $\chi$; We denote this representation by $\mathcal{W}$. The other irreducible $\mathbb{Q}$-representations of $G$ lift from those of $\mathbb{Z}/n\mathbb{Z}$.

\subsubsection*{Local monodromy $\bm{m} = (q, n, n)$}

Let $G := G_{q, n}$ as defined above in Section~\ref{sec:faithfulmetacyclic}. Throughout this section, we fix a $G$-Galois cover $C \rightarrow \mathbb{P}^1$ of genus $g$ with monodromy datum $(G, \bm{x})$ such that $\bm{m} = (q, n, n)$ and $x_1 = a$. This is the first case of Proposition~\ref{prop:monodromyfaithfulmetacyclic}. The Hurwitz equivalence class of $(G, \bm{x})$ is not uniquely determined, but the results in this section hold over all possible Hurwitz equivalence classes with these properties. We make the following classification:
\begin{proposition}[Theorem~\ref{thm:1.1}] 
\label{prop:4.2}
Let $C \rightarrow \mathbb{P}^1$ be the $G$-Galois cover with local monodromy $\bm{m} = (q, n, n)$ as defined above. Then, $H^0(C, \omega_C)$ may be computed with the Chevalley-Weil formula. As a consequence,
\begin{enumerate}[nolistsep, label={(\roman{enumi})}]
    \item The $\mathbb{Q}[G]$-module $H^1(C, \mathbb{Q})$ is isomorphic to the direct sum of $n - 2$ copies of $\mathcal{W}$ (the rational representation of $G$ with character $\chi$).
    \item The cover $C \rightarrow \mathbb{P}^1$ gives rise to a special point with $N = 0$ if and only if $n = 2, 3$; the $n = 2$ case is trivial.
\end{enumerate}
\end{proposition}
To apply the Chevalley-Weil formula, we need some elementary preliminaries (generalizing~\cite[Lemma 3.6]{Carocca2011}). For any integer $t$ let $[t]$ be the unique integer with $0 \le [t] \le q - 1, [t] \equiv t \pmod q$. For any integer $l, \ 1 \le l \le q - 1$ we define the set $O_l := \{l, [kl], \dots, [k^{n - 1} l] \}$, and denote by $S_l$ the sum of all elements in $O_l$. There exist a choice of $s$ representative $O_l$'s (and thus $l$'s) that form a partition of $\{1, \dots, q - 1\}$. We may number the representatives $l_1, \dots, l_s$.
\begin{lemma}
\label{basiclemma4.3}
The sum $S_l = l + [kl] + \dots + [k^{n - 1}l]$ (as defined above) is a multiple of $q$.
\begin{enumerate}[nolistsep, label={(\roman{enumi})}]
    \item If $n$ is even, then $S_l = \frac{n}{2} q$.
    \item \cite[Lemma 3.6]{Carocca2011} If $n = 3$, then $S_l$ is either $q$ or $2q$, taking on $q$ for exactly $\frac{q - 1}{2}$ values of $l$.
    \item If $n > 3$ is odd, then there exists $l$ such that $S_l, S_{q - l} > q$.
\end{enumerate}
\end{lemma}
\begin{proof}
Note $S_l \equiv l(1 + k + \dots + k^{n - 1}) \equiv l\frac{k^n - 1}{k - 1} \equiv 0 \pmod q$.

(i) If $n$ is even, then for any $0 \le j < \frac{n}{2}$, we have $[k^j l] + [k^{j + \frac{n}{2}} l] = [k^j l] + [-k^j l] = q$. There are $\frac{n}{2}$ such pairs $j, j + \frac{n}{2}$, implying the conclusion.

(ii) This is simply~\cite[Lemma 3.6]{Carocca2011}.

(iii) Since $[kl] + [k(q - l)] = q$, we have $S_l + S_{q - l} = nq$. Assuming on the contrary, we must have $\frac{s}{2}$ mutually disjoint subsets $O_{l_1}, \dots, O_{l_{\frac{s}{2}}} \subseteq \{1, \dots, q - 1\}$ such that $S_{l_i} = q$ (relabeling if necessary) for $1 \le i \le \frac{s}{2}$. Then, we have the bound
\begin{equation}
    \frac{q(q - 1)}{2n} = \sum_{i = 1}^{\frac{s}{2}} S_{l_i} \ge 1 + 2 + \dots + \frac{q - 1}{2} = \frac{(q + 1)(q - 1)}{8}
\end{equation}
and the RHS clearly exceeds the LHS if $n > 4$.
\end{proof}
Suppose $n$ is odd; then, the $l_i$'s may be partitioned into pairs where the sum of each pair is $q$. Thus, in what follows we assume the $l_i$'s are ordered such that $l_i + l_{i + \frac{s}{2}} = q$ for $1 \le i \le \frac{s}{2}$. Furthermore, take $t_i \equiv l_i \pmod q$ for all $i$. Thus, $\chi_i^{\ast} = \chi_{i + \frac{s}{2}}$ for $1 \le i \le \frac{s}{2}$.

We are now ready to use Lemma~\ref{basiclemma4.3} to prove Proposition~\ref{prop:4.2}.
\begin{proof}[Proof of Proposition~\ref{prop:4.2}]
By the Riemann-Hurwitz formula, the genus of $C$ is
\begin{equation}
\label{eq:4.7}
    g =  1 - \# G + \frac12 \left( \frac{\# G}{n}(q - 1) + \frac{\# G}{2}(n - 1) + \frac{\# G}{2}(n - 1) \right) = \frac{(n - 2)(q - 1)}{2}.
\end{equation}
Recall that the eigenvalues of the permutation matrix in~\eqref{eq:4.4} are $n$th roots of unity. 

Suppose $n$ is even. Then, the Chevalley-Weil formula yields for all $i$
\begin{equation}
    \mu_{\chi_i} = -n + \sum_{j = 0}^{n - 1} \left \langle - \frac{[k^j l_i]}{q} \right  \rangle + 2 \sum_{j = 1}^{n - 1} \left \langle -\frac{j}{n} \right \rangle = -1 + \frac{n}{2}.
\end{equation}
By counting dimensions, we conclude $(\frac{n}{2} - 1)(\mathcal{W} \otimes_{\mathbb{Q}} \mathbb{C}) = H^0(C, \omega_C)$, yielding the desired representation structure of $H^1(C, \mathbb{Q})$ by the Hodge decomposition. We have $H^1(C, \mathbb{Q}) = \{0 \}$ in case $n = 2$. Moreover, for $n > 2$, we have $\mu_{\chi_i} = \frac{n}{2} - 1 \ge 1$, implying $N > 0$ by Corollary~\ref{corr:3.4}. Notice, in particular, that the characters $\chi_i$ are real-valued.

Now, suppose $n$ is odd. Here, the Chevalley-Weil formula yields, for $1 \le i \le \frac{s}{2}$:
\begin{align}
\label{eq:4.9}
    \mu_{\chi_i} = -n + \sum_{j = 0}^{n - 1} \left \langle - \frac{[k^j l_i]}{q} \right  \rangle + 2 \sum_{j = 1}^{n - 1} \left \langle -\frac{j}{n} \right \rangle &= -1 + n - \frac{1}{q}S_{l_i}, \\
    \mu_{\chi_{i + \frac{s}{2}}} = -n + \sum_{j = 0}^{n - 1} \left \langle - \frac{[k^j (q - l_i)]}{q} \right  \rangle + 2 \sum_{j = 1}^{n - 1} \left \langle -\frac{j}{n} \right \rangle &= -1 + n - \frac{1}{q}S_{l_{i + \frac{s}{2}}}.
\end{align}
We notice $\mu_{\chi_i} + \mu_{\chi_{i + \frac{s}{2}}} = n - 2$, so by counting dimensions, we again conclude that $H^0(C, \omega_C)$ is completely determined by~\eqref{eq:4.9}. This also yields the desired structure of $H^1(C, \mathbb{Q})$. 

By Proposition~\ref{prop:jacobiancm}, the case $n = 3$ gives rise to a special point with $N = 0$. On the other hand, for $n > 3$, there exists $l_i$ such that $q < S_{l_i}, S_{l_{i + \frac{s}{2}}} < q(n - 1)$, so $\mu_{\chi_i}, \mu_{\chi_{i + \frac{s}{2}}} > 0$. The characters $\chi_i$ are complex-valued, so Corollary~\ref{corr:3.4} implies $N > 0$.
\end{proof}
\begin{example}
Take $G \simeq (\mathbb{Z}/31\mathbb{Z}) \rtimes (\mathbb{Z}/5\mathbb{Z})$ in the given situation. This group has $6$ irreducible complex-valued characters $\chi_1, \dots, \chi_6$ of degree $5$. Calculating the $O_{l_i}$'s and applying the method in the proof of Proposition~\ref{prop:4.2}, the multiplicities of $\chi_1, \dots, \chi_6$ are, permuted in increasing order, $0, 1, 1, 2, 2, 3$. Proposition~\ref{prop:Nformula} yields $N = 4$.
\end{example}
For the remainder of this section we fix $n = 3$. Thus, our $G$-cover $C \rightarrow \mathbb{P}^1$ has local monodromy $\bm{m} = (q, 3, 3)$. Assume the $l_i$'s are ordered such that $S_{l_i} = q$ for $1 \le i \le \frac{s}{2}$. Suppose $V_i$ is the $\mathbb{C}[G]$-module corresponding to $\chi_i$. Then, the proof of Proposition~\ref{prop:4.2} indicates
\begin{equation}
\label{eq:h0}
    H^0(C, \omega_C) = \bigoplus_{i = 1}^{\frac{s}{2}} V_i, \quad \overline{H^0(C, \omega_C)} = \bigoplus_{i = 1 + \frac{s}{2}}^{s} V_i.
\end{equation}
As a further consequence of Proposition~\ref{prop:4.2}, $\text{Jac}(C)$ consists of one isotypic component; namely, the one associated to $\mathcal{W}$. Hence, $\text{Jac}(C)$ has endomorphism by the $\mathbb{Q}$-algebra $M_3(\mathbb{Q}(\zeta_q^{(3)}))$, and so Proposition~\ref{prop:jacobiancm} implies it has CM by $\mathbb{Q}(\zeta_q)$, a totally real degree $3$ extension of $\mathbb{Q}(\zeta_q^{(3)})$ contained in $M_3(\mathbb{Q}(\zeta_q^{(3)}))$.

The CM-type of $\text{Jac}(C)$ with respect to $\mathbb{Q}(\zeta_q)$ was computed in~\cite[Corollary 3.9]{Carocca2011}. For completeness, we shall fill in the details of the result, using a different methodology from~\cite{Carocca2011}. In particular, we will make use of \textit{primitive central idempotents}, and the computational methods here will again be used in Section~\ref{sec:5}. To begin, let us check that the induced $\mathbb{Q}(\zeta_q)$-action on $H^1(C, \mathbb{Q})$ is compatible with the $\mathbb{Q}[G]$-action. Let $e := e_{\mathbb{Q}}(\chi_1)$ be the primitive central idempotent associated to $\chi$. Then,
\begin{equation}
    e = \frac{1}{\# G}\left( n(q - 1) - n \sum_{\nu = 1}^{q - 1} a^\nu \right) = \frac{q - 1}{q} - \frac{1}{q} \sum_{\nu = 1}^{q - 1} a^\nu,
\end{equation}
and this is exactly the primitive central idempotent associated to $\mathbb{Q}(\zeta_q)$ in the group algebra decomposition of $\mathbb{Q}[a] \simeq \mathbb{Q}[\mathbb{Z}/q\mathbb{Z}]$. It follows
\begin{equation}
    \mathbb{Q}(\zeta_q) \simeq \mathbb{Q}[a] e \subseteq \mathbb{Q}[G] e \simeq M_3(\mathbb{Q}(\zeta_q^{(3)})),
\end{equation}
where the first isomorphism may be specified by the embedding
\begin{equation}
    \mathfrak{i} : \mathbb{Q}(\zeta_q) \hookrightarrow \text{End}^0(\text{Jac}(C)), \quad \zeta_q \mapsto ae
\end{equation}
associated to the abelian variety $\text{Jac}(C)$ (the image of $\mathfrak{i}$ is $\mathbb{Q}[a]e$).

By~\eqref{eq:4.4} and~\eqref{eq:h0}, there exists a basis of $H^0(C, \omega_C)$ such that the $ae$-action is the matrix
\begin{equation}
    \text{diag} \left \{ \zeta_q^{t_1}, \zeta_q^{k t_1}, \zeta_q^{k^2 t_1}, \zeta_q^{t_2}, \dots , \zeta_q^{t_{\frac{s}{2}}}, \zeta_q^{k t_{\frac{s}{2}}}, \zeta_q^{k^2 t_{\frac{s}{2}}} \right \},
\end{equation}
which computes the CM-type of $\text{Jac}(C)$. In summary:
\begin{theorem}[Theorem~\ref{thm:1.1}]
\label{thm:4.6}
Let $C \rightarrow \mathbb{P}^1$ be the $G$-Galois cover with local monodromy $\bm{m} = (q, 3, 3)$ as defined above. Then, the Jacobian $(\text{Jac}(C), \mathfrak{i})$ (with $\mathfrak{i}$ given above) has CM by $(\mathbb{Q}(\zeta_q), \Phi)$, where the type $\Phi$ is given by the following $g$ embeddings
\begin{equation}
    \sigma : \zeta_q \mapsto \zeta_q^{k^j t_i}, \quad j = 0, 1, 2, \ 1 \le i \le \frac{s}{2}
\end{equation}
of $\mathbb{Q}(\zeta_q)$ into $\mathbb{C}$.
\end{theorem}
Note that we have recovered some CM Jacobian varieties arising from cyclic Galois covers, studied by Li et al.~\cite{li2018newton}:
\begin{proposition}
\label{prop:4.6}
Denote by $C \rightarrow \mathbb{P}^1$ the $G$-Galois cover with local monodromy $\bm{m} = (q, 3, 3)$ as previously defined and $C' \rightarrow \mathbb{P}^1$ the $\mathbb{Z}/q\mathbb{Z}$-cyclic cover with datum $\bm{y} = (1, k, k^2)$. Then, the Jacobians $\text{Jac}(C)$ and $\text{Jac}(C')$ are isogenous.
\end{proposition}
\begin{proof}
Using notation and terminology from~\cite[Section 2]{li2018newton}, it suffices to show that the signature type $\mathfrak{f}$ of $(\mathbb{Z}/q\mathbb{Z}, \bm{y})$---in this case another word for CM-type of $\text{Jac}(C')$---corresponds to the CM-type of $(\text{Jac}(C), \mathfrak{i})$. That is, it is enough to show $\mathfrak{f}(\tau_{k^j t_i}) = 1$ whenever $j = 0, 1, 2, \ 1 \le i \le \frac{s}{2}$ (automatically $\mathfrak{f}(\tau_{\nu}) + \mathfrak{f}(\tau_{q - \nu}) = 1$ for all $1 < \nu < q$). However, by the Chevalley-Weil formula (see~\cite[Equation 2.3]{li2018newton}, a restatement of Theorem~\ref{thm:cw}),
\begin{equation}
    \mathfrak{f}(\tau_{k^j t_i}) = -1 + \sum_{\nu = 0}^{2} \left \langle -\frac{k^{j + \nu} t_i}{q} \right \rangle = -1 + \sum_{\nu = 0}^{2} \left \langle -\frac{k^{\nu} l_i}{q} \right \rangle = 2 - \frac{1}{q} S_{l_i} = 1 \quad \text{if $1 \le i \le \frac{s}{2}$}.
\end{equation}
Notice this is the same expression as~\eqref{eq:4.9}. The result follows by the basic theory of CM abelian varieties (cf. Section~\ref{sec:2.4}).
\end{proof}
\begin{remark}
Notice by Theorem~\ref{thm:isotypicdec} and Proposition~\ref{prop:jacobiancm}, $\text{Jac}(C)$ is isogenous to $B^3$, where $B$ is an abelian subvariety of dimension $\frac{g}{3} = \frac{q - 1}{6}$. Then, $B$ is CM over the subfield $\mathbb{Q}(\zeta_q^{(3)})$ of index $3$ in $\mathbb{Q}(\zeta_q)$. This verifies that the Jacobian associated to the cyclic cover $C' \rightarrow \mathbb{P}^1$ in Proposition~\ref{prop:4.6} is \textit{not simple}; we would have not been able to obtain this result using only the theory of cyclic Galois covers, as in~\cite{li2018newton}.
\end{remark}

\section{Classification of Dicyclic Covers}
\label{sec:5}

In this section we extend the methods of Section~\ref{sec:4} to some groups with irreducible characters of higher Schur indices. To start, let $n > 1$ be a positive integer; we denote by $\text{Dic}_n$ the \textit{dicyclic group} of order $4n$:
\begin{equation}
\label{eq:dicall}
    \text{Dic}_n \simeq \langle x, y \ | \ x^{2n} = 1, x^n = y^2, y^{-1} x y = x^{-1} \rangle.
\end{equation}
In case $n$ is odd, we may write the semidirect product
\begin{equation}
\label{eq:dicq}    
    \text{Dic}_n \simeq (\mathbb{Z}/n\mathbb{Z}) \rtimes_2 (\mathbb{Z}/4\mathbb{Z}) \simeq \langle a, b \ | \ a^n = b^4 = 1, b^{-1} a b = a^{-1} \rangle.
\end{equation}
Here, the notation $\rtimes_2$ indicates that the map $\mathbb{Z}/4\mathbb{Z} \rightarrow \text{Aut}(\mathbb{Z}/n\mathbb{Z})$ has kernel isomorphic to $\mathbb{Z}/2\mathbb{Z}$. In Section~\ref{sec:dicyclicgroups} we study the case $G = \text{Dic}_q$ where $q$ is an odd prime. 

The case of the quaternion group $G = \text{Dic}_2 \simeq Q_8$:
\begin{equation}
\label{eq:quaterniongroup}
    Q_8 = \langle \mathbf{i}, \mathbf{j} \ | \ \mathbf{i}^2 = \mathbf{j}^2 = -1, \mathbf{ij} = -\mathbf{ji} \rangle
\end{equation}
is considered in Section~\ref{sec:quaterniongroup}. We use slightly different methods than the case $G = \text{Dic}_q$ noted above.

\subsection{Dicyclic Groups of Order \texorpdfstring{$4q$}{Lg}}
\label{sec:dicyclicgroups}

Throughout this section let $G := \text{Dic}_q$, where $q$ is an odd prime, as defined in~\eqref{eq:dicq}. We begin by noting the possible Hurwitz equivalence classes of $(G, \bm{x})$.
\begin{proposition}
\label{prop:monodromydicyclic}
Let $C \rightarrow \mathbb{P}^1$ be a $G$-Galois cover, branched at $3$ points, with monodromy datum $(G, \bm{x})$. Then, the local monodromy $\bm{m}$ is either $(q, 4, 4)$ or $(2q, 4, 4)$, and these correspond to all the possible distinct Hurwitz equivalence classes of data.
\end{proposition}
\begin{proof}
Notice that one SSG in $\bm{x}$ must lie in $\langle a \rangle$ or $b^2 \langle a \rangle$. In the first case, if $x_1 \in \langle a \rangle \setminus \{ 1 \}$, then $x_2, x_3$ must lie in the cosets $b \langle a \rangle$ and $b^3 \langle a \rangle$ in some order, in order for $x_1, x_2, x_3$ to generate $G$. By an appropriate $\mathbf{B}_3 \times \text{Aut}(G)$-action we may assume $x_1 = a, x_2 = b, x_3 = b^3 a^{-1}$.

In the second case, if $x_1 \in b^2 \langle a \rangle$, then notice $x_1 \neq b^2$; otherwise, $x_1 = b^2$ lies in the center $Z(G)$ so $x_1, x_2, x_3$ generates a commutative subgroup of $G$, a contradiction. Then, $x_2, x_3$ must both lie in either $b \langle a \rangle$ or $b^3 \langle a \rangle$. By an appropriate $\mathbf{B}_3 \times \text{Aut}(G)$-action we may assume $x_1 = ab^2 , x_2 = b, x_3 = b a^{-1}$.
\end{proof}

\subsubsection*{Representation Theory of $\text{{\normalfont Dic}}_q$ over $\mathbb{C}$ and its Subfields}

We shall review the representation theory of $G$ over $\mathbb{C}$ and its subfields $\mathbb{Q}, \mathbb{Q}(\zeta_4)$. First, there are $q + 3$ conjugacy classes of $G$, as follows. The center $Z(G)$ consists of $\{1\}, \{b^2\}$. There are $\frac{q - 1}{2}$ size $2$ conjugacy classes of $G$ contained in $\langle a \rangle$ of the form $\{a, a^{-1}\}, \{a^2, a^{-2}\}, \dots$. Likewise, there are $\frac{q - 1}{2}$ size $2$ conjugacy classes of $G$ contained in $b^2 \langle a \rangle$ of the form $\{b^2 a, b^2 a^{-1}\}, \{b^2 a^2, b^2 a^{-2}\}, \dots$. Finally, the cosets $b\langle a \rangle, b^3 \langle a \rangle$ are each conjugacy classes themselves. All in all, there are $q + 3$ conjugacy classes of $G$. 

Since $G' = \langle a \rangle$, we have $G/G' \simeq \mathbb{Z}/4\mathbb{Z}$, so $G$ has $4$ complex linear characters $\psi_0, \psi_1, \psi_2, \psi_3$, where $\psi_0$ is trivial, $\psi_1(b) = \zeta_4 = \sqrt{-1}, \psi_2(b) = -1, \psi_3(b) = \zeta_4^3 = -\sqrt{-1}$. Now, $G/Z(G) \simeq D_q$, and there are $s := \frac{q - 1}{2}$ degree-$2$ irreducible representations of $D_q$ that lift to degree-$2$ (irreducible) representations of $G$, each of Schur index $1$ over any subfield $F \subseteq \mathbb{C}$. Using notation from Section~\ref{sec:faithfulmetacyclic}, these representations are given by~\eqref{eq:4.4}, with $t_i = i$. 
We let $\chi_i$ be the character of $\rho_i$.
To find the remaining irreducible representations of $G$, consider the tensor product representations of the form $\psi_1 \otimes \rho_i$. One may compute the following matrix representations afforded by $\rho_i, \psi_1 \otimes \rho_i$:
\begin{align}
    \rho_i : a &\mapsto 
    \begin{bmatrix}
    0 & 1 \\
    -1 & \zeta_q^i + \zeta_q^{-i}
    \end{bmatrix}, \quad b \mapsto
    \begin{bmatrix}
    1 & -(\zeta_q^i + \zeta_q^{-i}) \\
    0 & -1
    \end{bmatrix} \nonumber \\
    \psi_1 \otimes \rho_i : a &\mapsto 
    \begin{bmatrix}
    0 & 1 \\
    -1 & \zeta_q^i + \zeta_q^{-i}
    \end{bmatrix}, \quad b \mapsto
    \begin{bmatrix}
    \zeta_4 & -(\zeta_q^i + \zeta_q^{-i})\zeta_4 \\
    0 & -\zeta_4 \label{eq:symplecticrep} 
    \end{bmatrix}.
\end{align}
Put $F_0 := \mathbb{Q}(\zeta_q + \zeta_q^{-1})$, the character field of each $\rho_i, \psi_1 \otimes \rho_i$ ($1 \le i \le \frac{q - 1}{2}$). Then,~\eqref{eq:symplecticrep} says $\psi_1 \otimes \rho_i$ may be realized over the field $F_0(\zeta_4)$, a quadratic extension of the character field $F_0$. Furthermore, the Schur index of each $\psi_1 \otimes \chi_i$ over the degree-$2$ extension $\mathbb{Q}(\zeta_4) \supset \mathbb{Q}$ is $1$. The representations $\psi_1 \otimes \rho_i$ are irreducible, i.e. by the irreducibility of $\rho_i$ and the fact $\psi_1$ is linear. Alternatively, if $\theta_j, \ 1 \le j \le q - 1$ is the linear representation of $\mathbb{Z}/2q\mathbb{Z} \simeq \langle b^2 a \rangle$ corresponding to the map $1 \mapsto \zeta_{2q}^j$, then notice $\text{Ind}_{\langle b^2 a \rangle}^G(\theta_{2i - 1})$ is isomorphic to $\psi_1 \otimes \rho_i$ for every $i$. Since each $\theta_j$ is linear with $\mathbb{Q}(\theta_j) = \mathbb{Q}(\zeta_{2q}) = \mathbb{Q}(\zeta_q)$, it follows $\psi \otimes \rho_i$ may be realized over $\mathbb{Q}(\zeta_q) = F_0(\zeta_q)$. This illustrates that the minimum degree extension of $F_0$ which realizes the character $\psi_1 \otimes \rho_i$ is not unique.
\begin{proposition}
\label{prop:6.3}
For any $i$, the character $\psi_1 \otimes \chi_i \in \text{Irr}(G)$ is symplectic. In particular, the Schur index of $\psi_1 \otimes \chi_i$ is $m_{\mathbb{R}}(\psi_1 \otimes \chi_i) = m_{\mathbb{Q}}(\psi_1 \otimes \chi_i) = 2$.
\end{proposition}
\begin{proof}
We have already shown $m_{\mathbb{Q}}(\psi_1 \otimes \rho_i) \le 2$. Since $m_{\mathbb{R}}(\psi_1 \otimes \rho_i) \le m_{\mathbb{Q}}(\psi_1 \otimes \rho_i)$, it suffices to prove $m_{\mathbb{R}}(\psi_1 \otimes \rho_i) = 2$. However, the Frobenius-Schur indicator (cf.~\eqref{eq:frobschur}) is
\begin{equation}
\begin{aligned}
    \iota_{\psi_1 \otimes \rho_i} &= \frac{1}{4q}\left(\sum_{\nu = 0}^{q - 1} \sum_{j = 0}^3 (\psi_1 \otimes \rho_i)((b^j a^{\nu})^2) \right) \\
    &= \frac{1}{4q} \left( 2 \cdot \left( \sum_{\nu = 0}^{q - 1} (\psi_1 \otimes \rho_i)(a^{2 \nu}) \right) + 2 \cdot \left( \sum_{\nu = 0}^{q - 1} (\psi_1 \otimes \rho_i)(b^2) \right) \right) = -1,
\end{aligned}
\end{equation}
where we have observed that $(b^j a^{\nu})^2 = a^{2 \nu}$ if $j$ is even and $(b^j a^{\nu})^2 = b^2$ if $j$ is odd. Thus, $\psi_1 \otimes \chi_i$ is a symplectic character. The results on the Schur index of $\psi_1 \otimes \chi_i$ follow immediately.
\end{proof}
By the results in Section~\ref{sec:representationtheory} on the Schur index, the rational group algebra decomposition is
\begin{equation}
    \mathbb{Q}[G] \simeq \mathbb{Q} \times \mathbb{Q} \times \mathbb{Q}(\zeta_4) \times M_2(F_0) \times \Delta_2(F_0),
\end{equation}
where $\Delta_2(F_0)$ is a degree-$4$ central division algebra over $F_0$. We let $\chi = \sum_{j = 1}^s \chi_j$, so that $\psi_1 \otimes \chi = \sum_{j = 1}^s (\psi_1 \otimes \chi_j)$. In everything that follows we denote by $\mathcal{W}^+$ and $\mathcal{W}^-$ the rational irreducible representations of $G$ whose respective characters are $\chi$ and $2(\psi_1 \otimes \chi)$; they correspond to the components $M_2(F_0)$ and $\Delta_2(F_0)$. Using the Schur indices of $G$ over $\mathbb{Q}(\zeta_4)$, we also obtain the group algebra $\mathbb{Q}(\zeta_4)[G]$:
\begin{equation}
    \mathbb{Q}[G] \otimes_{\mathbb{Q}} \mathbb{Q}(\zeta_4) \simeq  \mathbb{Q}(\zeta_4)[G] \simeq \mathbb{Q}(\zeta_4) \times \mathbb{Q}(\zeta_4) \times \mathbb{Q}(\zeta_4) \times \mathbb{Q}(\zeta_4) \times M_2(F_0(\zeta_4)) \times M_2(F_0(\zeta_4)),
\end{equation}
which is needed to handle the symplectic characters (cf. Theorem~\ref{prop:jacobiancmsymplectic}).

\subsubsection*{Local monodromy $\bm{m} = (q, 4, 4)$}

Let $G := \text{Dic}_q$ for $q$ an odd prime. Throughout this section, we fix a $G$-Galois cover $C \rightarrow \mathbb{P}^1$ of genus $g$ with monodromy datum $(G, \bm{x})$ such that $\bm{x} = (a, b, b^3 a^{-1}), \bm{m} = (q, 4, 4)$, corresponding to the first Hurwitz equivalence class identified by Proposition~\ref{prop:monodromydicyclic}.
\begin{proposition}[Theorem~\ref{thm:1.3}]
\label{prop:6.4}
Let $C \rightarrow \mathbb{P}^1$ be the $G$-Galois cover with local monodromy $\bm{m} = (q, 4, 4)$ as defined above. Then, $H^0(C, \omega_C)$ may be computed with the Chevalley-Weil formula. As a consequence,
\begin{enumerate}[nolistsep, label={(\roman{enumi})}]
    \item The $\mathbb{Q}[G]$-module $H^1(C, \mathbb{Q})$ is isomorphic to $\mathcal{W}^-$ (the rational irreducible representation of $G$ with character $2(\psi_1 \otimes \chi)$).
    \item The cover $C \rightarrow \mathbb{P}^1$ gives rise to a special point with $N = 0$.
\end{enumerate}
\end{proposition}
\begin{proof}
By the Riemann-Hurwitz formula, the genus of $C$ is $g = q - 1$. Note that the eigenvalues of $(\psi_1 \otimes \rho_i)(b)$ are $\zeta_4, \zeta_4^3$. The same is true for $b^3a^{-1}$. Then, the Chevalley-Weil formula yields for all $i$
\begin{equation}
\label{eq:5.13}
    \mu_{\psi_1 \otimes \chi_i} = -2 + \left \langle - \frac{i}{q} \right \rangle + \left \langle - \frac{-i}{q} \right \rangle + 2 \left( \left \langle - \frac{1}{4} \right \rangle + \left \langle - \frac{3}{4} \right \rangle \right) = 1.
\end{equation}
By dimension counting, the isomorphism class of $H^0(C, \omega_C)$ is completely determined by~\eqref{eq:5.13}. Since the character of $H^0(C, \omega_C)$ is $\psi_1 \otimes \chi$, and that $H^0(C, \omega_C)$ and $\overline{H^0(C, \omega_C)}$ are (non-canonically) isomorphic representations, we conclude the character of $H^1(C, \mathbb{Q})$, as a $\mathbb{Q}$-representation, is $m_{\mathbb{Q}}(\psi_1 \otimes \chi_1)(\psi_1 \otimes \chi) = 2(\psi_1 \otimes \chi)$, yielding (i). Result (ii) follows by Corollary~\ref{corr:3.4}.
\end{proof}
Since $H^1(C, \mathbb{Q})$ is a simple $\Delta_2(F_0)$-module, Theorem~\ref{thm:isotypicdec} implies that $\text{Jac}(C)$ has endomorphism by $\Delta_2(F_0)$. In this situation, Proposition~\ref{prop:jacobiancm} fails to apply, because there are Schur indices of $G$ greater than $1$. However, notice that $\text{Jac}(C) \sim e \text{Jac}(C)$, where $e := e_{\mathbb{Q}}(\psi_1 \otimes \chi_1) = e_{\mathbb{Q}(\zeta_4)}(\psi_1 \otimes \chi_1)$ is the primitive central idempotent associated to the symplectic character $\psi_1 \otimes \chi$ (over both $\mathbb{Q}$ and its extension $\mathbb{Q}(\zeta_4)$). Therefore, Theorem~\ref{prop:jacobiancmsymplectic} implies that $\text{Jac}(C)$ has endomorphism by $M_2(F_0(\zeta_4))$, induced by the $\mathbb{Q}(\zeta_4)[G]$-action. In particular, $\text{Jac}(C)$ has CM by $F_0(\zeta_4, \zeta_q) = \mathbb{Q}(\zeta_{4q}) = \mathbb{Q}(\zeta_4, \zeta_{2q})$ (note $F_0(\zeta_4, \zeta_q)$ is a maximal subfield of $M_2(F_0(\zeta_4))$). Moreover, Theorem~\ref{thm:isotypicdec} also implies the isogeny $\text{Jac}(C) \sim B^2$, where $B$ is an abelian subvariety of $\text{Jac}(C)$ with CM by $F_0(\zeta_4)$.

To find the CM type, we apply a similar computation as in Section~\ref{sec:faithfulmetacyclic}. That is, notice
\begin{equation}
\begin{aligned}
    e &= \frac{1}{\# G}\left(2(q - 1) - 2(q - 1) b^2 - 2\sum_{\nu = 1}^{s - 1} a^\nu + 2 \sum_{\nu \in (\mathbb{Z}/2q\mathbb{Z})^{\ast}} (b^2 a)^\nu \right) \\
    &= \frac{q - 1}{2q} - \frac{q - 1}{2q} b^2 + \frac{1}{2q} \sum_{\nu = 1}^{2s - 1} (-1)^\nu (b^2 a)^\nu,
\end{aligned}
\end{equation}
and this is exactly the primitive central idempotent associated to $\mathbb{Q}(\zeta_4, \zeta_{2q})$ in the group algebra decomposition of $\mathbb{Q}(\zeta_4)[b^2 a] \simeq \mathbb{Q}(\zeta_4)[\mathbb{Z}/2q\mathbb{Z}]$. It follows \begin{equation}
    \mathbb{Q}(\zeta_4, \zeta_{2q}) \simeq \mathbb{Q}(\zeta_4)[b^2 a] e \subseteq \mathbb{Q}(\zeta_4)[G] e \simeq M_2(F_0(\zeta_q)),
\end{equation}
where the first isomorphism may be specified by the embedding
\begin{equation}
    \mathfrak{i} : \mathbb{Q}(\zeta_4, \zeta_{2q}) \hookrightarrow \text{End}^0(\text{Jac}(C)), \quad \zeta_{2q} \mapsto b^2 a e
\end{equation}
associated to the abelian variety $\text{Jac}(C)$. Hence, the pre-image of $\zeta_4 b^2 ae$ is a primitive root of order $4q$, i.e. a generator of $\mathbb{Q}(\zeta_4, \zeta_{2q})$.

Since the eigenvalues of $(\psi_1 \otimes \rho_i)(b^2 a)$ are $-\zeta_q^i, -\zeta_q^{-i}$, there exists a basis of $H^0(C, \omega_C)$ such that the $b^2 ae$-action is the diagonal matrix given by $\text{diag}\{-\zeta_q, -\zeta_q^2, \dots, -\zeta_q^{q - 1}\}$. Now, by the above, $H^0(C, \omega_C)$ affords the scalar multiplication by $\zeta_4$, so we deduce that the $\zeta_4 b^2 ae$-action on $H^0(C, \omega_C)$ is the matrix
\begin{equation}
    \text{diag} \left \{ \zeta_4 \cdot \zeta_q, \zeta_4 \cdot \zeta_q^2, \dots, \zeta_4 \cdot \zeta_q^{q - 1}  \right \}
\end{equation}
or its complex conjugate. We summarize the main result in the following Theorem:
\begin{theorem}[Theorem~\ref{thm:1.3}]
\label{thm:cmtypedicyclic}
Let $C \rightarrow \mathbb{P}^1$ be the $G$-Galois cover with local monodromy $\bm{m} = (q, 4, 4)$ as defined above. Then, the Jacobian $(\text{Jac}(C), \mathfrak{i})$ (with $\mathfrak{i}$ given above) has CM by $(\mathbb{Q}(\zeta_{4q}), \Phi)$, where the type $\Phi$ is given by the following $g$ embeddings
\begin{equation}
    \sigma : \zeta_{4q} \mapsto \zeta_4 \cdot \zeta_q^j, \quad 1 \le j < q
\end{equation}
of $\mathbb{Q}(\zeta_{4q})$ into $\mathbb{C}$.
\end{theorem}

\subsubsection*{Local monodromy $\bm{m} = (2q, 4, 4)$}

Let $G := \text{Dic}_q$ as before. Throughout this section, we fix a $G$-Galois cover $C \rightarrow \mathbb{P}^1$ of genus $g$ with monodromy datum $(G, \bm{x})$ such that $\bm{x} = (ab^2 , b, b a^{-1}), \bm{m} = (2q, 4, 4)$, corresponding to the second Hurwitz equivalence class identified by Proposition~\ref{prop:monodromydicyclic}. The following theorem identifies that $H^1(C, \mathbb{Q})$ has both a complex-valued component and a symplectic component.

\begin{proposition}
\label{prop:6.9}
Let $C \rightarrow \mathbb{P}^1$ be the $G$-Galois cover with local monodromy $\bm{m} = (2q, 4, 4)$ as defined above. Then, $H^0(C, \omega_C)$ may be computed with the Chevalley-Weil formula. As a consequence,
\begin{enumerate}[nolistsep, label={(\roman{enumi})}]
    \item The $\mathbb{Q}[G]$-module $H^1(C, \mathbb{Q})$ is isomorphic to $\mathcal{W}^- \oplus \mathcal{U}$, where $\mathcal{U}$ is the rational irreducible representation of $G$ with character $\psi_1 + \psi_3$.
    \item The cover $C \rightarrow \mathbb{P}^1$ gives rise to a special point with $N = 0$.
\end{enumerate}
\end{proposition}
\begin{proof}
The proof is extremely similar to that of Proposition~\ref{prop:6.4}. By the Riemann-Hurwitz formula, the genus of $C$ is $g = q$. Notice, in particular, that the eigenvalues of $(\psi_1 \otimes \rho_i)(ab^2)$ are $-\zeta_q^i, -\zeta_q^{-i}$, a pair of complex conjugate primitive $2q$-th roots of unity. The Chevalley-Weil formula thus yields for all $i$ that $\mu_{\psi_1 \otimes \rho_i} = 1$. 

By the Riemann-Hurwitz formula (see~\eqref{eq:4.7}), the genus $g$ of $C$ is $q - 1$. Note that the eigenvalues of $(\psi_1 \otimes \rho_i)(b)$ are $\zeta_4, \zeta_4^3$. The same is true for $b^3a^{-1}$. Then, the Chevalley-Weil formula yields for all $i$
\begin{equation}
\label{eq:5.132}
    \mu_{\psi_1 \otimes \chi_i} = -2 + \left \langle - \frac{i}{q} \right \rangle + \left \langle - \frac{-i}{q} \right \rangle + 2 \left( \left \langle - \frac{1}{4} \right \rangle + \left \langle - \frac{3}{4} \right \rangle \right) = 1.
\end{equation}
The Chevalley-Weil formula also shows $\mu_{\psi_1} = 1, \mu_{\psi_3} = 0$. Hence, dimension counting implies that $H^0(C, \omega_C)$ is completely determined by~\eqref{eq:5.132} and $\mu_{\psi_1}, \mu_{\psi_3}$, Thus, the character of $H^1(C, \mathbb{Q})$, as a $\mathbb{Q}$-representation, is $m_{\mathbb{Q}}(\psi_1 \otimes \chi_1)(\psi_1 \otimes \chi) + (\psi_1 + \psi_3)$, yielding (i). Result (ii) follows by Corollary~\ref{corr:3.4}.
\end{proof}

Introduce the rational primitive central idempotents $e_0 := e_{\mathbb{Q}}(\psi_1 + \psi_3), e := e_{\mathbb{Q}}(\psi_1 \otimes \chi_1)$. By Proposition~\ref{prop:jacobiancm}, the isotypic component $e_0 \text{Jac}(C)$ has complex multiplication by $\mathbb{Q}(\zeta_4)$; the CM-type is primitive and automatically determined. As for $e \text{Jac}(C)$, we may repeat the argument of the $\bm{m} = (q, 4, 4)$ case without further loss. Note that the underlying complex vector space of holomorphic $1$-forms of $e \text{Jac}(C)$ is just $e H^0(C, \omega_C)$, a proper subspace of $H^0(C, \omega_C)$.
\begin{theorem}
\label{prop:symplecticcmpart2}
Let $C \rightarrow \mathbb{P}^1$ be the $G$-Galois cover with local monodromy $\bm{m} = (2q, 4, 4)$. Let $e$ be the rational primitive central idempotent defined above. Then, the isotypic part $e \text{Jac}(C)$ has endomorphism by $M_2(F_0(\zeta_4))$ induced by the $\mathbb{Q}(\zeta_4)[G]$-action. Therefore, $e \text{Jac}(C)$ affords the CM-pair $(\mathbb{Q}(\zeta_{4q}), \Phi)$, with $\Phi$ given by the following $g - 1$ embeddings
\begin{equation}
    \sigma : \zeta_{4q} \mapsto \zeta_4 \cdot \zeta_q^j, \quad 1 \le j < q.
\end{equation}
\end{theorem}
Thus, one recovers the CM Jacobian arising from the local monodromy $\bm{m} = (q, 4, 4)$ of the previous subsection.

\subsection{The Quaternion Group}
\label{sec:quaterniongroup}

Throughout this section let $G := Q_8$, using the presentation in~\eqref{eq:quaterniongroup}. Also, denote $\mathbf{k} := \mathbf{i}\mathbf{j}$. We shall see that the computation of the CM-field and type of $(G, \bm{x})$ data is similar, but slightly more sophisticated (and interesting) than the case of $G = \text{Dic}_q$ for $q$ an odd prime. We begin by noting that there is exactly one Hurwitz equivalence class of $G$-Galois covers of $\mathbb{P}^1$ with $3$ branch points:
\begin{proposition}
Let $C \rightarrow \mathbb{P}^1$ be a $G$-Galois cover branched at $3$ points. Then, this cover is Hurwitz equivalent to the datum $(G, \bm{x}), \ \bm{x} = (\mathbf{i}, \mathbf{j}, \mathbf{k})$.
\end{proposition}
\begin{proof}
This is easy to see, because any $x_i$ in a $3$-SSG $\bm{x}$ of $G \simeq Q_8$ clearly cannot be $\pm 1$.
\end{proof}
In the remainder of this section we shall fix the above datum, i.e. let $\bm{x} = (\mathbf{i}, \mathbf{j}, \mathbf{k})$ be a $3$-SSG of $G$ associated to the cover $C \rightarrow \mathbb{P}^1$.

As usual we need the representation theory of $G$. The group $G$ has four linear characters $\psi_1, \dots, \psi_4$ lifted from $\mathbb{Z}/2\mathbb{Z} \times \mathbb{Z}/2\mathbb{Z}$. The remaining character $\chi$ is degree $2$, where $\chi(1) = 2$, $\chi(-1) = -2$, and $\chi$ vanishes elsewhere. There exists a natural surjective homomorphism of $\mathbb{Q}$-algebras $\rho : \mathbb{Q}[G] \rightarrow \mathbb{H}_{\mathbb{Q}}$, where
\begin{equation}
    \mathbb{H}_{\mathbb{Q}} = \{a + b \mathbf{i} + c \mathbf{j} + d \mathbf{k}, \ a, b, c, d \in \mathbb{Q} \ | \ \mathbf{i}^2 = \mathbf{j}^2 = \mathbf{k}^2 = -1, \mathbf{ij} = -\mathbf{ji} = \mathbf{k} \}
\end{equation}
are the rational Hamilton quaternions, given by identifying elements of $G$ as generators of $\mathbb{H}_{\mathbb{Q}}$. Moreover, one checks $\rho(e(\chi)) = 1$, so that $\rho(\mathbb{Q}[G]e(\chi)) = \mathbb{H}_{\mathbb{Q}}$. Since $\mathbb{Q}[G]e(\chi)$ is simple, $\rho$ induces an isomorphism $\mathbb{Q}[G]e(\chi) \xrightarrow{\sim} \mathbb{H}_{\mathbb{Q}}$. The rational group algebra decomposition is thus
\begin{equation}
    \mathbb{Q}[G] \simeq \mathbb{Q} \times \mathbb{Q} \times \mathbb{Q} \times \mathbb{Q} \times \mathbb{H}_{\mathbb{Q}}
\end{equation}
and we have $m_{\mathbb{Q}}(\chi) = 2$. In fact, one checks via the Frobenius-Schur indicator~\eqref{eq:frobschur} that $\iota_{\chi} = -1$ and hence $m_{\mathbb{R}}(\chi) = 2$, so that $\chi$ is a symplectic character. Now, the complex irreducible representation associated to $\chi$ is isomorphic to the matrix representation $G \rightarrow GL_2(\mathbb{C})$ given by
\begin{equation}
\label{eq:matrepq8}
    \mathbf{i} \mapsto
    \begin{bmatrix}
    0 & -1 \\
    1 & 0 
    \end{bmatrix}, \quad
    \mathbf{j} \mapsto
    \begin{bmatrix}
    0 & -\sqrt{-1} \\
    -\sqrt{-1} & 0
    \end{bmatrix}, \quad
    \mathbf{k} \mapsto
    \begin{bmatrix}
    \sqrt{-1} & 0 \\
    0 & -\sqrt{-1}
    \end{bmatrix}.
\end{equation}
Thus, $\chi$ has Schur index $1$ over $\mathbb{Q}(\zeta_4)$, so that the group algebra $\mathbb{Q}(\zeta_4)[G]$ is
\begin{equation}
    \mathbb{Q}[G] \otimes_{\mathbb{Q}} \mathbb{Q}(\zeta_4) \simeq  \mathbb{Q}(\zeta_4)[G] \simeq \mathbb{Q}(\zeta_4) \times \mathbb{Q}(\zeta_4) \times \mathbb{Q}(\zeta_4) \times \mathbb{Q}(\zeta_4) \times M_2(\mathbb{Q}(\zeta_4)).
\end{equation}
\begin{proposition}
\label{prop:6.1}
Let $G \rightarrow \mathbb{P}^1$ be a $G$-Galois cover with monodromy datum $(G, \bm{x})$. Then, the character of $H^0(C, \omega_C)$ is $\chi$, and $H^1(C, \mathbb{Q})$ is a rational irreducible representation whose character is $2 \chi$. Thus, the datum $(G, \bm{x})$ gives rise to a special point with $N = 0$.
\end{proposition}
\begin{proof}
The Riemann-Hurwitz formula implies that the genus of the curve $C$ arising from $(G, \bm{x})$ is $g = 2$. Since the eigenvalues of the matrices in~\eqref{eq:matrepq8} are $\zeta_4, \zeta_4^3$, the Chevalley-Weil formula yields
\begin{equation}
    \mu_{\chi} = -2 + 3 \cdot \left( \left \langle -\frac14 \right \rangle + \left \langle -\frac34 \right \rangle \right) = 1,
\end{equation}
implying that the character of $H^0(C, \omega_C)$ is $\chi$. Now, $\chi$ is real-valued, so the associated complex conjugate representation is isomorphic to itself, so the character of $H^1(C, \mathbb{C})$ is $2\chi$. The corresponding rational irreducible representation is irreducible by Schur index considerations. That $(G, \bm{x})$ corresponds to a special point in $\mathsf{A}_g$ is immediate by Corollary~\ref{corr:3.4}.
\end{proof}
Hence, the Jacobian $\text{Jac}(C)$ of the cover $C \rightarrow \mathbb{P}^1$ has complex multiplication, not dissimilar to the previous cases in Section~\ref{sec:5}. In particular, denote by $e$ the primitive central idempotent associated to the symplectic character $\chi$. Then, $\text{Jac}(C) \sim e \text{Jac}(C)$ and by Theorem~\ref{prop:jacobiancmsymplectic}, the Jacobian $\text{Jac}(C)$ has endomorphism by $M_2(\mathbb{Q}(\zeta_4))$ induced by the $\mathbb{Q}(\zeta_4)[G]$-action. It has CM by $\mathbb{Q}(\zeta_4) \times \mathbb{Q}(\zeta_4)$, where each $\mathbb{Q}(\zeta_4)$ component corresponds to a primitive idempotent of $M_2(\mathbb{Q}(\zeta_4))$. Here we have implicitly used the fact that $\text{Jac}(C) \sim B^2$, where $B$ is an abelian subvariety of $\text{Jac}(C)$ with CM by $\mathbb{Q}(\zeta_4)$. In fact, each $B$ is an elliptic curve. Note that the CM of the $Q_8$-Galois cover is in contrast to that of the $\text{Dic}_q$-Galois cover: the latter is simple, while the former contains two simple components.

Now we find the CM type of $\text{Jac}(C)$. Notice
\begin{equation}
    e = \frac12 - \frac12\mathbf{k}^2.
\end{equation}
Let
\begin{equation}
    e_0 = \frac14 - \frac{\zeta_4}{4}\mathbf{k} - \frac14 \mathbf{k}^2 + \frac{\zeta_4}{4}\mathbf{k}, \quad e_1 = \frac14 + \frac{\zeta_4}{4}\mathbf{k} - \frac14 \mathbf{k}^2 - \frac{\zeta_4}{4}\mathbf{k}
\end{equation}
be the distinct primitive central idempotents associated to the faithful characters $\pi_0, \pi_1$ of $\langle \mathbf{k} \rangle \simeq \mathbb{Z}/4\mathbb{Z}$, that is, $\pi_0 : \mathbf{k} \mapsto \zeta_4, \pi_1 : \mathbf{k} \mapsto -\zeta_4$. Then, since $e_0 + e_1 = e$, we have that 
\begin{equation}
    \mathbb{Q}(\zeta_4) \times \mathbb{Q}(\zeta_4) \simeq \mathbb{Q}(\zeta_4)[\mathbf{k}](e_0 + e_1) = \mathbb{Q}(\zeta_4)[\mathbf{k}]e \subseteq \mathbb{Q}(\zeta_4)[G]e \simeq M_2(\mathbb{Q}(\zeta_4)),
\end{equation}
where the first isomorphism may be specified with the embedding
\begin{equation}
    \mathfrak{i} : \mathbb{Q}(\zeta_4) \times \mathbb{Q}(\zeta_4) \hookrightarrow \text{End}^0(\text{Jac}(C)), \quad (\zeta_4, 0) \mapsto \mathbf{k}e_0, \ (0, \zeta_4) \mapsto \mathbf{k}e_1.
\end{equation}
Fix a basis of $H^0(C, \omega_C)$ consistent with the matrix representation~\eqref{eq:matrepq8}. Then, under this matrix representation, we compute
\begin{equation}
    \mathbf{k} e_0 = \begin{bmatrix}
    \zeta_4 & 0 \\ 0 & 0
    \end{bmatrix}, \quad
    \mathbf{k} e_1 = \begin{bmatrix}
    0 & 0 \\ 0 & -\zeta_4
    \end{bmatrix}.
\end{equation}
This encodes the action of each $\mathbb{Q}(\zeta_4)$ component on the associated simple component of $\text{Jac}(C)$. We conclude the main result:
\begin{theorem}
Let $C \rightarrow \mathbb{P}^1$ be a $G$-Galois cover with datum $\bm{x}$ as defined above. Then, $(\text{Jac}(C), \mathfrak{i})$ (with $\mathfrak{i}$ defined above) is CM. The CM-pair of $\text{Jac}(C)$ may be given as a disjoint union 
\begin{equation}
    (\mathbb{Q}(\zeta_4), \Phi_0) \sqcup (\mathbb{Q}(\zeta_4), \Phi_1).
\end{equation}
The two sub-types are the embeddings $\Phi_0 = \{\mathbb{Q}(\zeta_4) \hookrightarrow \mathbb{C}, \zeta_4 \mapsto \zeta_4\}, \ \Phi_1 : \{\mathbb{Q}(\zeta_4) \hookrightarrow \mathbb{C}, \zeta_4 \mapsto -\zeta_4\}$.
\end{theorem}
%


\section*{Acknowledgements}

The research for this project was conducted in the summer and fall of 2022, during the author’s SURF (Summer Undergraduate Research Fellowship) at Caltech. The author would like to thank the Caltech Student Faculty Programs Office and the Carl F. Braun Residuary Trust for supporting this research financially. The author would also like to thank Elena Mantovan for their invaluable support of this research during the SURF and for providing helpful comments on this paper.

\printbibliography

\end{document}